\documentclass[12pt]{amsart}
 \usepackage{amssymb, color,amsfonts,amsthm,amscd,verbatim, mathtools,amsmath}
\usepackage[square, comma, sort&compress, numbers]{natbib}
 \usepackage{graphicx}
 \usepackage{bm,comment}
\newtheorem{thm}{Theorem}[section]
\newtheorem{lem}{Lemma}[section]

\newtheorem{pro}{Proposition}[section]

\newtheorem{ex}[thm]{Example}

\theoremstyle{definition}
\newtheorem{dfn}[lem]{Definition}
\theoremstyle{remark}
\newtheorem{rem}{Remark}[section]
\numberwithin{equation}{section}

\def\R{{\Bbb R}}
\let\cal=\mathcal

\begin{document}

\title[Systematic Measures of Biological Networks, Part II]{Systematic Measures of Biological Networks, Part II:
Degeneracy, Complexity and Robustness}
\author[Y. Li] {Yao Li}
\address{Y. Li: Department of Mathematics and Statistics, University
  of Massachusetts Amherst, MA, 01003, USA}
\email{yaoli@math.umass.edu}

\author[Y. Yi]{Yingfei Yi}
\address{Y. Yi: Department of Mathematical
\& Statistical Sci, University of Alberta, Edmonton, Alberta,
Canada T6G 2G1,  School of Mathematics, Jilin University, Changchun
130012, PRC, and School of Mathematics, Georgia Institute of
Technology, Atlanta, GA 30332, USA} \email{yingfei@ualberta.ca,
yi@math.gatech.edu}

\thanks {This research  was partially supported by NSF grant DMS1109201. The second author is also partially supported by NSERC discovery grant 1257749, a faculty development grant from University of Alberta, and a Scholarship from Jilin University. }



\begin{abstract} This paper is Part II of a two-part series devoting to the
study of systematic measures in a complex bio-network modeled by a
system of ordinary differential equations.  In this part, we
quantify several systematic measures of a biological network
including degeneracy, complexity and robustness. We will apply the
theory of stochastic differential equations to define degeneracy and
complexity for a bio-network. Robustness of the network will be
defined according to the strength of attractions to the global
attractor. Based on the study of stationary probability measures and
entropy made in Part I of the series, we will investigate some
fundamental properties of these systematic measures, in particular
the connections between degeneracy, complexity and robustness.
\end{abstract}

\maketitle

\section{Introduction}


Consider a biological network modeled by the following system of ordinary
differential equations (ODE system for short):
\begin{equation}
   \label{ODE1}
   x' = f(x),\;\;\;x\in \R^n,
\end{equation}
where $f$ is a $C^1$ vector field on $\R^n$, called {\em drift
field}.  Adopting the idea of activating  the functional connections
among modules of the network via external noises in the case of
neural systems \cite{tononi1999measures, li2012quantification},
we add  additive white noise perturbations $\sigma dW_{t}$ to
\eqref{ODE1} to obtain the following system of stochastic
differential equations (SDE system for short):
\begin{equation}
   \label{SDE1}
   \mathrm{d}X = f(X) \mathrm{d}t + \epsilon \sigma(x) \mathrm{d}W_{t},\; \;\; X\in \R^n,
\end{equation}
where  $W_t$ is the standard $m$-dimensional Brownian
 motion,  $\epsilon$ is a small parameter lying in an interval $(0,\epsilon^*)$, and  $\sigma$, called an {\em noise matrix}, is
 an  $n\times m$ matrix-valued, bounded, $C^1$ function on $\R^n$  for some positive integer $m\ge n$,
 such that $\sigma(x)\sigma^{\top}(x)$ is everywhere
non-singular. We denote the collection of such noise
  matrices by $\Sigma$. Under certain dissipation conditions, the SDE system
\eqref{SDE1} generates a diffusion process in $\R^n$ with
well-defined transition probability kernel, and moreover, if the
transition probability kernel admits a density function
$p^t(\xi,x)$, then its time evolution $ u(x,t) =
\int_{\mathbb{R}^{n}} p^{t}(z, x) \xi(z) \mathrm{d}z $ satisfies the
Fokker-Planck equation (FPE for short):
\begin{equation}
   \label{FPE1}
  \left\{\begin{array}{l}

\frac{\partial u(x,t)}{\partial t}=\frac{1}{2}{{\epsilon
}^{2}}\sum\limits_{i,j=1}^{n}{\partial_{ij}(a_{ij}(x)u(x) )-
\sum_{i=1}^{n}\partial_{i}(f(x) u(x)
   )} := L_{\epsilon}u(x),\\
   \int_{\R^n}u(x)\rm dx=1,\end{array}\right.
\end{equation}
where $(a_{ij}(x)) := A(x) := \sigma (x){{\sigma }^{\top}}(x)$.
Denote
\[
   \mathcal{L}_{\epsilon} =
   \frac{1}{2}\epsilon^{2}\sum_{i,j=1}^{n}a_{ij}(x)\partial_{ij} +
   \sum_{i=1}^{n} f_{i}(x)\partial_{i}
\]
as the  adjoint of Fokker-Planck operator. If $u(x)$ is a weak
stationary solution of \eqref{FPE1}, i.e., $u$ is a strictly
positive, continuous function on $\R^n$ with $\int_{\R^n} u(x) {\rm
d} x=1$ such that
\begin{equation}\label{un}
  \int_{\R^n} \mathcal{L}_{\epsilon}h(x) u(x) \mathrm{d}x = 0,
    \qquad\; \forall h \in C_{0}^{\infty}( \mathbb{R}^{n}),
\end{equation}
then the probability measure $\mu_{\epsilon}(\rm d x) = u(x) \mathrm{d} x$ is clearly a stationary  measure of \eqref{FPE1}, i.e.,
\begin{equation}\label{mun}
  \int_{\R^n} \mathcal{L}_{\epsilon} h(x) \mu_{\epsilon}(\mathrm{d}x) = 0,
    \qquad\; \forall h \in C_{0}^{\infty}( \mathbb{R}^{n}).
\end{equation}
Conversely, it follows from the regularity theory of stationary
measures \cite{bogachev2001regularity} that any stationary measure
of \eqref{FPE1} must admit a density function which is necessarily a
weak stationary solution of \eqref{FPE1}.  We remark
that an invariant probability measure of the diffusion process
generated from SDE \eqref{SDE1} must be a stationary measure of the
FPE \eqref{FPE1}  and vice versa under some conditions.

In Part I of the series, we have assumed the following conditions:
\medskip

\begin{itemize}
\item[{\bf H$^0$)}] System \eqref{ODE1} is dissipative
 and there exists a strong Lyapunov function
$W(x)$  with respect to an isolating neighborhood $\mathcal{N}$
of the global attractor $\mathcal{A}$ such that
\begin{displaymath}
   W(x) \geq L_{1} \mathrm{dist}^{2}(x, \mathcal{A}),\;\; x \in \mathcal{N}
\end{displaymath}
for some $L_{1} > 0$.
\end{itemize}

\medskip

\begin{itemize}
  \item[{\bf H$^1$)}] For each $\epsilon \in (0, \epsilon^{*})$, the Fokker-Planck equation
  \eqref{FPE1} admits a unique stationary probability measure
    $\mu_{\epsilon}$ such that  for an isolating
    neighborhood $\cal N$ of $\cal A$,
$$
 \lim_{\epsilon \rightarrow 0} \frac{\mu_{\epsilon}( \mathbb{R}^{n}
   \setminus \mathcal{N})}{\epsilon^{2}}  = 0,
$$
and moreover, there are constants $p, R_0>0$ such that
$$
  \mu_{\epsilon}( \{ x \,:\, |x| > r \}) \leq e^{-\frac {r^{p}}{\epsilon^{2}}}
$$
for all $r>R_0$ and all $\epsilon \in (0, \epsilon^{*})$.
\end{itemize}
\medskip
The desired concentration in {\bf H$^{1}$)} can follow
from various conditions, such as the existence of a quasi-potential
function or a suitable Lyapunov function. See Part I of the series and
Proposition 2.1 below for more information in this regard.

\medskip

 For each given $\epsilon\in
(0,\epsilon^*)$,  the mutual information $MI(X_{1};X_{2})$ among any
two modules (coordinate subspaces) $X_1,X_2$ can be defined using
the margins $\mu_{1}$, $\mu_{2}$   of $\mu_{\epsilon}$ with respect
to $X_{1}, X_{2}$, respectively. Such mutual information can then be
used to quantify degeneracy and complexity. Inspired by
\cite{tononi1999measures}, we will define the
$\{\epsilon,\sigma\}$-degeneracy and -complexity of the evolutionary
network \eqref{ODE1} associated with $\sigma$ as an averaged
combinations of certain mutual informations between different
modules. Let $\{I,\mathcal{O}\}$ be a pair of coordinate subspaces of
the variable set $\R^n$ which decompose $\R^n$, called an {\em
input-output pair}.  For any $0\leq k \leq |I|$, where $|I|$ denotes
the dimension of the input space $I$, the degeneracy
$D_\epsilon(I_{k})$ and complexity $C_\epsilon(I_{k})$,
associated with the $k$-decomposition $I = I_{k}\cup I_{k}^{c}$ is
defined as
\[
 D_\epsilon(I_{k})=MI(I;{{I}_{k}};\mathcal{O})=MI({{I}_{k}};\mathcal{O})+MI(I_{k}^{c};
\mathcal{O})-MI(I;\mathcal{O})
\]
and
\[
C_\epsilon(I_{k})=MI({{I}_{k}};I_{k}^{c})\,,
\]
where $I_{k}$ is a $k$-dimension subspace of $I$ spanned by $k$
variables. The degeneracy $D_\epsilon(\mathcal{O})$, respectively
complexity $C_\epsilon(\mathcal{O})$,
 with respect to the input-output pair
$\{I,\mathcal{O}\}$ is simply the average of all
$D_\epsilon(I_{k})$'s, respectively all  $C_\epsilon(I_{k})$'s. The
{\em degeneracy}, respectively {\em complexity}, of the network
\eqref{ODE1}  associated with $\sigma$, is then defined as $\mathcal{D}_{\sigma}=\liminf_{\epsilon\rightarrow 0} \sup_{\mathcal O}D_\epsilon(\mathcal{O})$,
respectively $ \mathcal{C}_{\sigma}=\liminf_{\epsilon\rightarrow 0} \sup_{\mathcal
O}C_\epsilon(\mathcal{O})$. We refer the readers to Section 3 for
details.

Another systematic measure for the network \eqref{ODE1} is the {\em
robustness}, which will be defined in Section ~\ref{robustness}
relevant to  the strength of its global  attractor, either in a
uniform way or in an average way. As suggested in
\cite{kitano2007towards, kitano2004biological}, the robustness is not always equivalent to
the stability. As to be seen in Section ~\ref{robustness}, if the
performance function of the network \eqref{ODE1} is known, then one
can also define its  functional robustness.

Many simulations and experiments have already suggested that there
are close connections among  degeneracy, complexity and robustness
in a biological system (see e.g. \cite{edelman2001degeneracy,
stelling2004robustness, whitacre2010degeneracy,
whitacre2012degeneracy, clark2011degeneracy}).  For the evolutionary network \eqref{ODE1}
and its noise perturbation \eqref{SDE1}, we will rigorously show the
following results under the conditions {\bf H$^0$)} and {\bf H$^1$) }:
\medskip

\begin{enumerate}
\item[1.]  {\em With respect to a fixed  $\sigma\in \Sigma$, high degeneracy always yields high
  complexity} (Theorem ~\ref{deglarger}).
\item[2.] {\em A robust system with
  non-degenerate attractor has positive degeneracy}  with respect to any  $\sigma\in \Sigma$ (Theorem
  ~\ref{twisted}).
\item[3.] {\em A robust system with stable equilibrium has  positive degeneracy
  with respect to any  $\sigma\in \Sigma$  under certain algebraic conditions} (Theorem ~\ref{degfixpt}).
\end{enumerate}
\medskip

 As  in \cite{edelman2001degeneracy} for neural systems, results above
 are useful in characterizing degenerate biological networks in connection with their
system complexities. This series of papers serves as a mathematical
  supplement of \cite{li2012quantification}.  We refer readers to
\cite{li2012quantification} for degeneracy, complexity,
  and robustness in biological models and discussions in this
  regard. Examples in \cite{li2012quantification} include a signaling pathway
network and a population model.

The paper is organized as follows. Section 2 is a preliminary
section. Section 3 defines degeneracy and
complexity. The robustness is investigated in Section 4. Finally, the connection between degeneracy,
complexity and robustness are proved in Section 5.

\section{Preliminary}

\subsection{Existence and concentration of stationary measures}
It was shown in Part I of the series \cite{li2014systematic} that the condition {\bf H$^1$)} is
implied by {\bf H$^0$)} together with the following condition:
\medskip

\begin{itemize}
\item[{\bf H$^{2}$)}] There is a positive function $U\in C^2(\R^n\setminus \cal
A)$ satisfying the following properties:
\begin{itemize} \item[{\rm
i)}] $\lim_{|x|\to\infty} U(x)=\infty$; \item[{\rm ii)}] There
exists a constant $\rho_m>0$  such that $U$ is a uniform Lyapunov
function of the family \eqref{FPE1} of class ${\cal B}^*$ in ${\cal
N}_\infty=:\mathbb{R}^{n}\setminus \Omega_{\rho_{m}}(U)$, i.e.,
there is a constant $\gamma>0$ independent of $\epsilon$ such that
\[
 \mathcal{L}_{\epsilon} U(x) < -\gamma, \;\qquad x\in {\cal
 N}_\infty
\]
for all $\epsilon\in (0,\epsilon^*)$. Moreover, a
function $H(\rho) \in L^{1}_{loc}([\rho_{m}, \infty))$ and constants
$p>0$, $R>\rho_m$ exist such that
\begin{eqnarray*}
 && H(\rho) \geq |\nabla U(x)|^{2},\quad x \in \Gamma_{\rho}(U),\\
 &&
  \int_{\rho_m}^{\rho}\frac{1}{H(s)} \mathrm{d}s \geq  |x|^{p}, \quad x
  \in \Gamma_{\rho}(U)
\end{eqnarray*}
for all $\rho > R$;
\item[{\rm iii)}] There exists a constant $\bar\rho_m\in (0,\rho_m)$ such
that $U$ is a uniform weak Lyapunov function of the family
\eqref{FPE1} in ${\cal N}_*=:\R^n\setminus {\cal N}_\infty\setminus
\Omega_{\bar\rho_{m}}(U)$, i.e.,
\[
  \mathcal{L}_{\epsilon} U(x) \leq 0, \;\qquad x\in {\cal N}_*
\]
for all $\epsilon\in (0,\epsilon^*)$;
\item[{\rm  iv)}] $\nabla U(x)\ne 0$,
$x\in \R^n\setminus  \Omega_{\bar\rho_{m}}(U)$;
\item[{\rm v)}] $\Omega_{\bar\rho_m} (U)\subset {\cal N}$.
\end{itemize}
\medskip

In the above, ${\cal L}_\epsilon$, $\epsilon\in (0,\epsilon^*)$,
is the adjoint Fokker-Planck operator and $\Gamma_\rho$,
$\Omega_\rho(U)$ denote the $\rho$-level set, $\rho$-sublevel set of
$U$ for each $\rho>0$ respectively.
\end{itemize}

\medskip

In summary, we have the following result.
\medskip
\begin{pro}~\label{cor3.1} {\rm  (Corollary~3.1, \cite{li2014systematic}  )} Conditions {\bf H$^0$), H$^{2}$)} imply {\bf H$^1$)}.
\end{pro}

\medskip

\begin{thm}
\label{accurate} {\rm (Theorem 3.1, \cite{li2014systematic}) } If
both {\bf H$^0$)} and {\bf H$^1$) } hold, then   for any $0<\delta
\ll 1$ there exist constants $\epsilon_{0},M>0$  such that
\begin{displaymath}
   \mu_{\epsilon}(B( \mathcal{A},M\epsilon)) \geq 1-\delta,
\end{displaymath}
whenever $\epsilon\in (0,\epsilon_{0})$.
\end{thm}

\medskip
\begin{thm}
\label{MSD} {\rm (Theorem 3.3, \cite{li2014systematic}) } Let
$$
  V(\epsilon) = \int_{\mathbb{R}^{n}} \mathrm{dist}^{2}(x,
  \mathcal{A}) \mu_{\epsilon}( \mathrm{d}x) \,.
$$
If both {\bf H$^0$)} and {\bf H$^1$) }
hold, then there are constants $V_{1}, V_{2}, \epsilon_{0} > 0$ such
that
$$
  V_{2}\epsilon^{2} \leq V(\epsilon) \leq V_{1}(\epsilon) \ , \quad
  \epsilon \in (0,  \epsilon_{0} )\,.
$$
\end{thm}

Let $\mu$ be the probability measure with density $u$, define the {\it
differential entropy} by
$$
\mathcal{H}(\mu) = -\int_{\mathbb{R}^{n}} u(x) \log u(x) \mathrm{d}x \,.
$$

\begin{thm}
\label{EntDimThm} {\rm (Theorem 4.1, \cite{li2014systematic}) }
Assume that {\bf H$^{0}$)} and {\bf H$^{1}$)} hold. If $\mathcal{A}$
is a regular set, then
\begin{equation}
   \label{EntDim}
   \liminf_{\epsilon\rightarrow
     0}\frac{\mathcal{H}(\mu_{\epsilon})}{\log \epsilon} \geq n - d \,,
\end{equation}
where $d$ is the Minkowski dimension of $\mathcal{A}$. If in addition
the family $\{\mu_{\epsilon}\}$ is regular with respect to
$\mathcal{A}$, then the equality holds in \eqref{EntDim}.
\end{thm}


For the definition of regular sets and measures, see Section 2.3 for
the detail.

\subsection{Tightness} For a Borel set $\Omega\subset\R^n$, let
$M(\Omega)$ denote the set of Borel probability measures on $\Omega$
furnished with the {\em weak$^*$-topology}, i.e., $\mu_k\to \mu$ iff
\[
\int_{\Omega} f(x) \mathrm{d} \mu_{k} (x) \to \int_{\Omega} f(x) \mathrm{d}
\mu(x),
\]
for every $f\in C_b(\Omega)$. A subset $\mathcal{M}\subset
M(\Omega)$ is said to be {\em tight} if for any $\epsilon>0$ there
exists a compact subset $K_\epsilon \subset \Omega$ such that
$\mu(\Omega \setminus K_\epsilon)<\epsilon$ for all $\mu\in
\mathcal{M}$.
\medskip

\begin{thm}\label{Prokh} {\rm (Prokhorov's Theorem, \cite{dellacherie1978probabilities})}   If a subset
$\mathcal{M}\subset M(\Omega)$ is tight, then it is relatively
sequentially compact in $M(\Omega)$.
\end{thm}
\medskip

\subsection{Regularity of sets and measures}
A set $A \subset \mathbb{R}^{n}$ is called a {\it regular set} if
\begin{displaymath}
   \limsup_{r \rightarrow 0} \frac{\log m(B(A,r))}{-\log r}  = \liminf_{r\rightarrow 0}
   \frac{\log m(B(A,r))}{-\log r} = n-d
\end{displaymath}
for some $d\geq 0$. Hereafter, $m(\cdot)$ denotes the Lebesgue
measure on $\mathbb{R}^{n}$.  It is easy to check that
  $d$ is the Minkowski dimension of $A$. Regular sets form a large class that
includes smooth manifolds and some fractal sets like Cantor sets.
However, not all measurable sets are regular.

Assume that \eqref{ODE1} admits a global attractor $\cal A$ and the
Fokker-Planck equation \eqref{FPE1} admits a  stationary probability measure
$\mu_\epsilon$  for each  $\epsilon\in (0,\epsilon_*)$. The family
$\{\mu_\epsilon\}$
 of
stationary probability measures  is said to be {\it regular with respect to
$\mathcal{A}$}  if for any $\delta > 0$ there are constants $K$, $C$
and a family of approximate funtions $u_{K,\epsilon}$ supported on $B( \mathcal{A}, K\epsilon)$ such that for all $\epsilon \in (0, \epsilon^*)$,
\begin{itemize}
  \item[a)] \begin{equation}\label{regular}
   \inf_{B(\mathcal{A},K\epsilon) }(u_{K, \epsilon}(x))\geq C
   \sup_{B(\mathcal{A},K\epsilon)}(u_{K, \epsilon}(x))  \,;
\end{equation}
and
\item[b)]
$$
 \|u_{\epsilon}(x) - u_{K, \epsilon}(x) \|_{L^{1}}  \leq \delta \,,
$$
where $u_{\epsilon}$ is the density function of $\mu_{\epsilon}$.
\end{itemize}

\medskip

Part I \cite{li2014systematic} gives several examples of regular
family $\mu_{\epsilon}$ with respect to $\mathcal{A}$. We conjecture
that the family $\mu_{\epsilon}$ is regular with respect to
$\mathcal{A}$ for a much larger class of systems. Details will be
given in our future work.



\subsection{2-Wasserstein metric}
Originally introduced in the study of optimal transportation
problems, the 2-Wasserstein metric is a distance function
 for probability distributions on a given metric space.
Let  $\mathcal{P}(\R^n)$ be the set of probability
measures on $\R^n$ with  finite second moment. The {\it
2-Wasserstein distance} $\mathcal{W}(\mu,\nu)$ between two
probability measures $\mu, \nu \in \mathcal{P}(\R^n)$ is defined by
\begin{displaymath}
   \mathcal{W}^2(\mu,\nu) = \inf_{r \in
     \mathcal{P}(\mu,\nu)}\int_{\R^n\times \R^n} |x-y|^2 \mathrm{d}r,
\end{displaymath}
where $\mathcal{P}(\mu,\nu)$ is the set of all probability measures
on the space $\R^n\times \R^n$ with marginal $\mu$ and $\nu$.
Intuitively, $\mathcal{W}(\mu,\nu)$ measures the minimum ``cost'' of
turning measure $\mu$ to measure $\nu$. The topology  on
$\cal P(\R^n)$ defined by the 2-Wasserstein metric is
essentially the same as the weak$^*$ topology on
$\mathcal{P}(\R^n)$.
\medskip

\begin{thm}
{\rm (Theorem 7.1.5, \cite{ambrosio2006gradient})} For a given
sequence $\{\mu_{n}\} \subset \mathcal{P}(X)$, \break
$\lim_{n\rightarrow \infty} \mathcal{W}(\mu_{n},\mu) = 0$ if and
only if $\mu_{n}\to \mu$ under the weak$^*$ topology and second
moments of $\{\mu_{n}\}$ are uniformly bounded.
\end{thm}
\medskip

 Given $\mu,\nu\in {\cal P}(\R^n)$, a measure $r$ on
$\mathcal{P}(\R^n \times \R^n)$ is called the {\it optimal measure}
if $r\in \mathcal{P}(\mu,\nu)$ and
\begin{displaymath}
    \mathcal{W}^2(\mu,\nu) =\int_{\R^n\times \R^n} |x-y|^2 \mathrm{d}r.
\end{displaymath}
The set of optimal measures with respect to $\mu, \nu\in {\cal
P}(\R^n)$ is denoted by $\mathcal{P}_{0}(\mu,\nu)$.

 The variational problem in finding the optimal measure
is called the {\it Kantorovich problem}, which,  under certain
regularity conditions, is equivalent to the so-called  {\it Monge
problem} of finding a measurable map  $T:\R^n\to \R^n$, called a
{\em transport map}, such that
\begin{displaymath}
\mathcal{W}^{2}(\mu,\nu) = \inf_{T\sharp\mu = \nu}\int_{\R^n} |x -
T(x)|^{2} \mathrm{d}x \,,
\end{displaymath}
where $T\sharp \mu$ stands for the push-forward map.
\medskip

\begin{thm}
\label{Monge} {\rm (Theorem 6.2.4, \cite{ambrosio2006gradient})}
 Suppose that  $\mu,\nu \in \mathcal{P}(\mathbb{R}^{n})$
with $\mu$ being Borel regular and
\begin{displaymath}
   \mu(\{x\in \mathbb{R}^{n} : \int_{\R^n} |x-y|^{2} \nu( \mathrm{d}y) < \infty
   \}) > 0,
\end{displaymath}
\begin{displaymath}
   \nu(\{x\in \mathbb{R}^{n} : \int_{\R^n}  |x-y|^{2} \mu( \mathrm{d}y) < \infty
   \}) > 0.
\end{displaymath}
Then there exists a unique optimal measure $r$,  and moreover,
\begin{displaymath}
   r = (i \times T)\sharp \mu
\end{displaymath}
for some transport map $T$ with $T\sharp\mu = \nu$, where $i$ is the
identity map on $\R^n$.
\end{thm}

\subsection{Estimates of differential entropy}
Let $u_{\epsilon}(x)$ be the probability density function of
$\mu_{\epsilon}$.

\begin{lem}
\label{entout} {\rm (Lemma 4.1, \cite{li2014systematic}) } Let $l >
0$ be a constant independent of $\epsilon$. If {\bf H$^{1}$)} holds,
then there exist positive constants $\epsilon_{0}, R_{0}$ such that
$$
  \int_{|x| > R_{0}} u_{\epsilon}(x) \log u_{\epsilon}(x) \geq
  -\epsilon^{l}, \quad \epsilon \in (0, \epsilon_{0}) \,.
$$
\end{lem}

\begin{lem}
\label{entin} {\rm (Lemma 4.2, \cite{li2014systematic}) } Let $v(x)$
be a probability density function on $\mathbb{R}^{n}$. Let $\Omega$ be
a Lebesgue measurable compact set. Then there is
a constant $\delta_{0} > 0$ such that for each $\delta \in (0,
\delta_{0})$, if
$$
  \int_{\Omega} v(x) \mathrm{d}x \leq \delta \,,
$$
then
$$
  \int_{\Omega} v(x) \log v(x) \mathrm{d}x \geq -2\sqrt{\delta}.
$$
\end{lem}

\begin{lem}
\label{entbound} {\rm (Lemma 4.3, \cite{li2014systematic}) } If {\bf
H$^{1}$)} holds, then there is a constant $\epsilon_{0} > 0$ such
that $u_{\epsilon}(x) \leq \epsilon^{-(2n+1)}$ whenever $x \in
\mathbb{R}^{n}$ and $\epsilon \in (0, \epsilon_{0})$.

In addition,
there are positive constants $R_{0}$ and $p$ such that
$$
  u_{\epsilon}(x) \leq e^{-|x|^{p}/2\epsilon^{2}}
$$
for any $|x| > R_{0}$ and $\epsilon \in (0, \epsilon_{0})$.
\end{lem}

\section{Degeneracy and complexity}\label{dandc}

In this section, we give quantitative definitions of degeneracy and
complexity for a biological network modeled by a system of ordinary
differential equations. Some fundamental properties of these
quantities will  be investigated.

\subsection{Quantifying degeneracy and complexity} We first define
 degeneracy and  complexity for a SDE system \eqref{SDE1} with
respect to  a fixed $\epsilon$ and a fixed noise matrix $\sigma$.
Let $\epsilon$ and  $\sigma$ be fixed in \eqref{SDE1} and assume
that  the corresponding Fokker-Planck equation \eqref{FPE1} admits a
unique stationary measure  $\mu=\mu_{\epsilon,\sigma}$. It follows
from the regularity theorem in \cite{bogachev2001regularity} that  $\mu$ admits a
density function which we denote by $u(x)$, $x\in \mathbb{R}^n$.

Let $I$ be a coordinate subspace, i.e., a subspace of $\R^n$ spanned
by some of the standard unit vectors $\{e_1,\cdots,e_n\}$.  Denote
$J$ as the orthogonal complement of $I$.  If $x_1$, $x_2$ denote the
coordinates of $I$, $J$ respectively, then the marginal distribution
with respect to $I$ reads
    \[{{u}_{I}}(x_1)=\int_{J} u(x_{1}, x_{2}) \text{d}x_{2},\]
and we can define  the {\em projected entropy on $I$}   by
    \[ H(I)=-\int_{I}{{u_{I}}}(x_{1})\log {u_{I}}(x_{1})dx_1,\]
which roughly measures  the uncertainty (amount of information) of
the $I$-component of  the random variable generated by \eqref{SDE1}.

For any two such coordinate subspaces ${{I}_{1}},{{I}_{2}}$, since
$H({{I}_{1}}\oplus {{I}_{2}})=H({{I}_{2}}\oplus {{I}_{1}})$, we
can define this quantity as the {\em joint entropy} between $I_1$
and $I_2$, denoted in short by $ H({{I}_{1}},{{I}_{2}})$. The {\em
mutual information among subspaces ${{I}_{1}},{{I}_{2}}$} is defined
by

    \[M({{I}_{1}};{{I}_{2}})=  H({{I}_{1}})+
    H({{I}_{2}})- H({{I}_{1}},{{I}_{2}}).\]
It is easy to see that
\begin{equation}
\label{minf}
 MI({{I}_{1}};{{I}_{2}})=\int_{{{I}_{1}}\oplus
          {{I}_{2}}}{{{u }_{{{I}_{1}},{{I}_{2}}}}}(x_{1},x_{2})\log
        \frac{{{u }_{{{I}_{1}},{{I}_{2}}}}(x_{1},x_{2})}{{{u
            }_{{{I}_{1}}}}(x_{1}){{u
            }_{{{I}_{2}}}}(x_{2})}\text{d}x_{1}\text{d}x_{2} \,.
\end{equation}
 Statistically, the mutual information \eqref{minf}
measures the correlation between marginal distributions with respect
to subspaces $I_1$ and $I_2$.

\medskip

Now let $\mathcal{O}$ be a fixed coordinate subspace of $R^n$,
viewed as an {\em output set}, and  $I$ be the orthogonal complement
of $\mathcal{O}$, viewed as the {\em input set}. To measure the
noise impacts on all possible components of the input set, we
consider an arbitrary  $k$-dimensional coordinate subspace $I_{k}$
of $I$ and denote its orthogonal complement in $I$ by $I_{k}^{c}$.
The {\em multivariate mutual information}, or the {\em interacting
information} among ${{I}_{k}}$, $I_{k}^{c}$ and $\mathcal{O}$ is
defined by
\begin{equation}
\label{degeneracy}
 MI(I_{k}; I_{k}^{c};\mathcal{O}) :=MI({{I}_{k}};\mathcal{O})+MI(I_{k}^{c};
\mathcal{O})-MI(I;\mathcal{O}) \,.
\end{equation}
Note that if $k = 0$, we have $MI(I_{k};
  I_{k}^{c};\mathcal{O}) = 0$. We refer readers to
  \cite{yeung2002first} for further properties of the multivariate
  mutual information.

Similar to the case of neural systems studied in
\cite{tononi1999measures}, we define the {\em degeneracy associated
with $\mathcal{O}$} by averaging all the multivariate mutual information
among all possible coordinate subspaces of $I$, i.e.,
\begin{equation}
\label{degeneracy2} D(\mathcal{O})=\langle
MI(I_{k};I_{k}^{c},\mathcal{O})\rangle := \sum_{0\leq k \leq |I|}
\frac{1}{2
  {|I| \choose k}}\max
        \{MI(I_{k}; I^{c}_{k};\mathcal{O}),0\} \,.
\end{equation}
Similarly, the {\it complexity \index{Complexity} $C(\mathcal{O})$
associated with $\mathcal{O}$} is defined by averaging all the
mutual information between ${{I}_{k}}$ and $I_{k}^{c}$, i.e.,
\begin{equation}
 \label{complexity}
C(\mathcal{O})=\langle MI({{I}_{k}};I_{k}^{c})\rangle =\sum_{0 \leq
k \leq |I|} \frac{1}{2 {|I| \choose k}} MI({{I}_{k}};I_{k}^{c}).
\end{equation}
For a biological network, the complexity measures how much the
co-dependency in a network appears among different modules rather
than different elements.

However, differing from the case of neural system, output sets in an
(evolutionary) biological network modeled by a system of ODEs are varying.
This motivates the following definition.

\begin{dfn}
1) For fixed diffusion matrix $\sigma $ and $\epsilon >0$, the {\em
$\{\sigma,\epsilon\}$-degeneracy ${{\mathcal{D}}_{\epsilon ,\sigma
}}$} and  {\em $\{\sigma,\epsilon\}$-complexity
${{\mathcal{C}}_{\epsilon ,\sigma }}$} of the system \eqref{ODE1}
are defined by

 \[{{\mathcal{D}}_{\epsilon ,\sigma
   }}=\underset{\mathcal{O}}{\max}\,D(\mathcal{O}),\]
\[{{C}_{\epsilon
     ,\sigma }}=\underset{\mathcal{O}}{\max}\,C(\mathcal{O}).\]

2) For fixed diffusion matrix $\sigma $, the {\em
$\sigma$-degeneracy ${{\mathcal{D}}_{\sigma }}$} and (structural)
{\em $\sigma$-complexity ${{\mathcal{C}}_{\sigma }}$} of the system
\eqref{ODE1} are defined by

 \[\mathcal{D}_{\sigma} =
\liminf_{\epsilon\rightarrow 0} \mathcal{D}_{\epsilon,\sigma},\]
\[\mathcal{C}_{\sigma} =
\liminf_{\epsilon\rightarrow 0} \mathcal{C}_{\epsilon,\sigma}.\]

3) The {\em degeneracy} $\mathcal D$ and the (structural) {\em
complexity} $\mathcal C$ of system \eqref{ODE1} are defined by
\[
  \mathcal{D} = \sup_{\|\sigma\| = 1} \mathcal{D}_{\sigma},
\]
\[
  \mathcal{C} = \sup_{\|\sigma\| = 1} \mathcal{C}_{\sigma}.
\]

4) We call a differential system \eqref{ODE1} {\em
$\sigma$-degenerate} (resp. {\em $\sigma$-complex}) with respect to
a perturbation matrix $\sigma $ if there exists ${{\epsilon
  }_{0}}$, such that ${{\mathcal{D}}_{\epsilon ,\sigma }}>0$
(resp. ${{\mathcal{C}}_{\epsilon ,\sigma }}>0$) for all $0<\epsilon
<{{\epsilon }_{0}}$. The system \eqref{ODE1} is said to be {\em
degenerate} (resp. {\em complex}) if $\mathcal D>0$ (resp. $\mathcal
C>0$).

\end{dfn}

\begin{rem}  1) A common output set is necessary to
    quantify the degeneracy. Inspired by \cite{tononi1999measures}, we
    use multivariate information to measure how much more correlation
    the inputs $I_k$ and $I_k^c$ share with output $\cal O$ than
    expected. Biologically, the multivariate mutual information $MI(I_{k},
    I_{k}^{c}, \mathcal{O})$ measures how much $I_k$ and $I_k^c$ are
    structurally different but perform the same function at the output set
    $\mathcal{O}$. Similarly, by taking the average over all possible
    decomposition of the input set, $\mathcal{D}(\mathcal{O})$ measures the
    ability of structurally different components in a network to
    perform similar function on designated output set.

2) The purpose of injecting external fluctuation is to detect
interactions among the network. When the injected noise at distinct
directions are not independent, the measured interactions
(degeneracy) may be polluted by the correlations among the external
fluctuations. See Remark 5.1 for further discussion. Hence in application, we usually adopt additive white
noise, i.e., let $\sigma = Id$ and study $D_{Id}$.

3) In biological applications, one can estimate the degeneracy (in
various meanings above) by selecting suitable output space as the
natural space containing  ``observable'' elements (see \cite{tononi1999measures} for
an example of a signaling network).

4) We remark that degeneracy and complexity depends on the choice of
coordinate systems. Both degeneracy and complexity measure the statistical
dependence between modules of networks. This statistical dependence is
determined by both dynamics of underlying equations and the choice of
observables. A change of coordinates means a change of the observables, which may affect the statistical dependence between modules
of observables. For example random variables $X_{1} + X_{2}$ and $X_{1} - X_{2}$ may have
a strictly positive mutual information even if $X_{1}$ and $X_{2}$ are
independent. In application, we usually use the natural coordinates
which is generated by nodes of networks.

\end{rem}

\subsection{Persistence of degeneracy and complexity}

The following lemma gives bounds of projected density function.

\begin{lem}
\label{projbounds} Assume {\bf H$^{1}$)} holds and let $u_{I}$ be
the projected density function onto a coordinate subspace $I$. Then
there exist positive numbers $\epsilon_{0}$, $p$ and $R$, such that
for any $\epsilon \in (0, \epsilon_{0})$,  $u_{I}(x_1) <
e^{-|x_1|^{p}/2\epsilon^{2}} \ \mathrm{ when} \ |x_1| > R$ and
$u_{I}(x_1) < \epsilon^{-(2n+2)}\ \mathrm{ when}\ |x_1| \leq R$.
\end{lem}

\begin{proof}
Since $u_{I}$ is the projection of $u$, $u_{I}$ has the same tail as
$u$. More precisely, it follows from {\bf H$^{1}$)} that there are
constants $p_{0}, R_{0} > 0$ such that
$$
  \int_{I \setminus B(0, r)} u_{I}(x_1) \mathrm{d}x_1 <
  e^{-|r|^{p_{0}}/\epsilon^{2}}
$$
for all $r > R_{0}$ and all $\epsilon \in (0, \epsilon^{*})$. By
Lemma ~\ref{entbound}, there exist positive numbers $\epsilon_{1}$,
$p$ and $R$, such that $u_{I}(x_1) < e^{-|x_1|^{p}/2\epsilon^{2}} \
\mathrm{ as } \ |x_1| > R$, for all $\epsilon \in (0,
\epsilon_{1})$, where $R = R_{0} + 1$.

\medskip

Using Lemma ~\ref{entbound}  one can make $\epsilon$ sufficiently
small such that $u(x) < \epsilon^{-(2n+1)}$ for all $x \in B(0, R)$.
Then it is easy to see from the definition of $u_I$  that
$$
  u_{I}(x_1) \leq C(R)\epsilon^{-(2n+1)} + \int_{J} e^{-|x_2|^{p}/2\epsilon^{2}} \mathrm{d}x_2
$$
for all $|x_1| \le R$, where $C(R)$ is the volume of ball with
radius $R$ in $J$. Hence for sufficient small $\epsilon$
 $u_{I}(x_1)$ is smaller than $\epsilon^{-(2n+2)} $ as $|x_1| \le R$.

\end{proof}

We now give the result  below concerning the persistence of
degeneracy and complexity.

\begin{thm}
\label{entper} Let $f_{l}$, $l \geq 1$ be a sequence of drift fields
such that $f_{l} \rightarrow f$ uniformly in $C^{2}$ norm. For any
fixed $0<\epsilon \ll 1$, denote the $\epsilon, \sigma$-degeneracy
with respect to \eqref{SDE1} with drift fields $f_{l}$ and $f$ by
$\mathcal{D}^{l}_{\epsilon, \sigma}$ and $\mathcal{D}_{\epsilon,
  \sigma} $ respectively. If condition {\bf H$^{1}$)} is uniformly satisfied by
equations \eqref{FPE1} with drift fields $\{f_{l}\}_{l \geq 1}$ and
$f$, then
$$
\lim_{l \rightarrow \infty}  \mathcal{D}^{l}_{\epsilon, \sigma} =
  \mathcal{D}_{\epsilon, \sigma}.
$$
\end{thm}

\begin{proof}

Denote the stationary probability measure of equation \eqref{FPE1}
with drift fields $\{f_{l}\}$ and $f$ by $\mu_{l}$ and $\mu$
respectively. Denote $u_{l}$ and $u$ as the corresponding density
functions.

Since {\bf H$^{1}$)} is uniformly satisfied, it is easy to see that the
sequence $\{\mu_{l}\}$ is tight. By Theorem 2.4,
$\{\mu_{n}\}$ is sequentially compact in the space of probability
measures on $\mathbb{R}^{n}$ equipped with the weak-* topology. We
note that each $\mu_l$ satisfies
\begin{equation}\label{mun1}
  \int_{\R^n} \mathcal{L}_{\epsilon}h(x) \mu_{l}( \mathrm{d}x) = 0,
    \qquad\; \forall h \in C_{0}^{\infty}( \mathbb{R}^{n}).
\end{equation}
Let $\mu_{*}$ be a limit point of $\{\mu_{l}\}$ and  $ \{\mu_{l_k}
\}$ be a subsequence of $\{\mu_{l}\}$ that converges to $\mu_{*}$
weakly. Since $\{f_{l}\}$ are uniformly bounded and $h \in
C^{\infty}_{0}( \mathbb{R}^{n})$, applying the dominated convergence
theorem to \eqref{mun1} shows that $\mu_{*}$ is the stationary
probability measure of \eqref{FPE1}. It follows from  the uniqueness
of stationary probability measure that $\mu = \mu_{*}$.
Consequently, $\mu_l$ converges to $\mu$ weakly as $l\to\infty$. It
follows that $u_l\to u$, as $l\to\infty$, pointwise in $\R^n$.

By Lemma ~\ref{entbound},  one can make $\epsilon$  sufficiently
small such that both $u(x)$ and $u_{n}(x)$ are  bounded from above
by
\[
M(x)=\left\{\begin{array}{ll} \epsilon^{-(2n+1)} , &\quad
\hbox{if} |x|<R_{0};\\
e^{-|x|^{p}/2\epsilon^{2}}, & \quad\hbox{if} |x| \geq R_{0},
\end{array}\right.
\]
 where $R_{0}$ and $p$ are constants in
{\bf H$^{1}$)}. Since $|x\log x|$ is increasing on both intervals
$(0,e^{-1})$ and $(1, +\infty)$, it is easy to see that $|u(x)\log
u(x)| \leq |M(x)\log M(x)| + M(x)$, $|u_{l}(x)\log u_{l}(x)| \leq |M(x)\log
M(x)|$ and
$$
  \int_{\mathbb{R}^{n}} \left ( |M(x) \log M(x) | + M(x) \right )\mathrm{d}x < \infty \,.
$$
Hence the dominated convergence theorem yields that
$$
 \lim_{l\rightarrow \infty} \int_{\mathbb{R}^{n}} u_{l}(x) \log
 u_{l}(x) \mathrm{d}x = \int_{\R^n} u(x) \log u(x)
 \mathrm{d}x \,.
$$
For any coordinate  subspace $I$ of $\mathbb{R}^{n}$, a similar
argument and Lemma 3.3 shows that
$$
 \lim_{l\rightarrow \infty} \int_{I} (u_{l})_{I}(x_1) \log
 (u_{l})_{I}(x_1) \mathrm{d}x_1 = \int_{I} u_{I}(x_1) \log u_{I}(x_1)
 \mathrm{d}x_1 \,.
$$

The theorem now follows easily from the definitions of
$\mathcal{D}^{l}_{\epsilon,\sigma}$ and
$\mathcal{D}_{\epsilon,\sigma}$.
\end{proof}

Theorem ~\ref{entper} only holds for  fixed $\epsilon$ and $\sigma$.
We will see in Section 5 that even for fixed $\sigma$, the
continuous dependence of $\sigma$-degeneracy on $f$ will require
additional conditions.

\section{Robustness}
\label{robustness}

In this section, we introduce and discuss various notions of
robustness for a global attractor of an ODE system from different
perspectives, which  can be used as useful systematic measures of a
biological network. These notions will be
introduced to measure the strength of attraction of the global
attractor because a stronger attractor tends to have a  better
ability to remain stable under noise perturbations.

\subsection{Uniform Robustness} Uniform robustness describes
the uniform attracting strength of  the global attractor $\cal A$ of
system \eqref{ODE1}.

Assume that $\cal A$ is a strong attractor, i.e.,  there is a
neighborhood $\cal N$ of $\cal A$, called an {\it isolating neighborhood},
a smooth function $U$ on $\cal N$, called a strong Lyapunov
function, and a constant $\gamma_{0}>0$, called Lyapunov function,
such that
  $\nabla U(x)\ne 0$, $x\in  \cal N\setminus \cal A$, and
$$
  f(x) \cdot \nabla U(x) \leq - \gamma_{0} |\nabla U(x)|^{2}, \;\qquad x \in
  N \setminus \cal A.
$$
Any nonnegative  constant $\alpha$ such that
 \[
 \frac{\nabla U(x)}{|\nabla U(x)|}\cdot f(x)\le -\alpha
        \text{dist}(x,\mathcal{A}),\ \; {\rm  \forall} \ x\in \mathcal{N}
        \]
is called an {\it index of $\cal A$ associated with $U$} or simply
an {\it index of $\cal A$} (note that $\alpha$ depends on both
choices of $\cal N$ and $U$).

\begin{dfn}
For a strong attractor $\mathcal{A}$ with index $\alpha$, the {\it
uniform robustness}\index{Uniform Robustness} of the strong
attractor $\mathcal{A}$ is the following quantity
 \[R_{u}=\sup \{ \alpha  : \, \alpha \ \hbox{ is an index of} \
        \mathcal{A}\}.\]

The system \eqref{ODE1} is said to be {\it robust} if $\mathcal{A}$
is a strong attractor and $R_{u} > 0$.
\end{dfn}

\medskip
\begin{pro}\label{prop4.1} If {\bf H$^{0}$)} holds, then the system \eqref{ODE1} is robust.
\end{pro}
\begin{proof} The proposition follows easily from {\bf H$^{0}$)} and the definitions of
strong attractor and robustness.

\end{proof}

\subsection{2-Wasserstein Robustness}

Let $\mathcal{P}(\mathbb{R}^{n})$ denote the space of probability
measures on $\mathbb{R}^{n}$, endowed with the 2-Wasserstein metric
$d_{w}$. In the case of weak$^*$ convergence of $\mu_\epsilon$, as
$\epsilon\to 0$, the 2-Wasserstein distance between $\mu_\epsilon$
and its weak limit measure measures  certain  averaged persistence property of
$\cal A$ under the stochastic perturbations. We note from
\cite{huang5} that the limit of $\mu_\epsilon$ must be an
invariant measure of \eqref{ODE1} supported on $\cal A$.

\begin{dfn} The {\em 2-Wasserstein robustness} (or {\em average robustness}) $R_w$ of \eqref{ODE1} w.r.t. $\sigma$ is defined as the reciprocal of metric
derivative, i.e.,
\begin{displaymath}
   R_{w} = \inf_{\mu_{0} \in \mathcal{M}, \epsilon_{n} \rightarrow 0} \left \{\lim_{n \rightarrow \infty}
     \frac{\epsilon_{n}}{ \mathcal{W}( \mu_{\epsilon_{n}}, \mu_{0}) }
     \quad : \quad \mu_{\epsilon_{n}} \rightarrow \mu_{0} \mbox{ weakly as }
     \epsilon_{n} \rightarrow 0
   \right \} \,,
\end{displaymath}
where $\mathcal{M}$ is the set of sequential limit point of
$\{\mu_{\epsilon}\}$ as $\epsilon \to 0$. The system \eqref{ODE1}
is said to be {\it robust in the 2-Wasserstein sense w.r.t.
$\sigma$} if $R_{w} >0$.
\end{dfn}

Roughly speaking, 2-Wasserstein robustness gives the first order
expansion of $\mu_{\epsilon}$ in terms of $\epsilon$ in the
2-Wasserstein metric spaces.

\medskip

\begin{thm}
\label{w2MSD} If {\bf H$^{0}$)} and {\bf H$^{1}$)} hold, then $R_w$ is finite.
\end{thm}
\begin{proof}
Without loss of generality, we assume that $R_w>0$. Then
$\mu_\epsilon$ converges to an invariant measure  $\mu_{0}$ of
\eqref{ODE1}, and it follows from  \cite{huang5} that
$\mathrm{supp}(\mu_{0})\subset \mathcal{A}$. Hence $\mu_{\epsilon}$
and $\mu_{0}$ satisfy conditions of Theorem ~\ref{Monge}.

\medskip

By Theorem ~\ref{Monge}, $\mathcal{W}^{2}(\mu_{\epsilon}, \mu_{0})$
solves the following Monge problem
\begin{displaymath}
\mathcal{W}^{2}(\mu_{\epsilon},\mu_{0}) = \inf_{T\sharp\mu_{\epsilon}
  = \mu_{0}}\int_{\R^n} |x - T(x)|^{2} \mathrm{d}x \,.
\end{displaymath}
Since $\mu_{0}$ is supported in $\mathcal{A}$, $T(x) \in
\mathcal{A}$ whenever $T\sharp \mu_{\epsilon} = \mu_{0}$. Therefore
\begin{displaymath}
   |x-T(x)|^{2} \geq \mathrm{dist}^{2}(x,\mathcal{A})
\end{displaymath}
for any map $T : \mathbb{R}^{n}\rightarrow \mathbb{R}^{n}$ that
satisfies $T\sharp \mu_{\epsilon} = \mu_{0}$. It follows that
\begin{equation}\label{WW}
   \mathcal{W}^{2}(\mu_{\epsilon},\mu_{0}) \geq \int_{\R^n}
   \mathrm{dist}^{2}(x,\mathcal{A}) \mu_{\epsilon}(\mathrm{d}x) \,.
\end{equation}

\medskip

By Theorem ~\ref{MSD}, there are positive constants $V_{2}$ and
$\epsilon_{0}$ such that
$$
  \int_{\R^n}
   \mathrm{dist}^{2}(x,\mathcal{A}) \mu_{\epsilon}(\mathrm{d}x)  \geq V_{2}\epsilon^{2}
$$
for all $\epsilon \in (0, \epsilon_{0})$. Thus as $\epsilon$
approaches zero, the mean square displacement is bounded
from below by $V_{1}\epsilon^{2}$. Hence $R_{w}$ is finite by definition.
\end{proof}

\subsection{Functional Robustness}

The robustness of a biological system is not completely equivalent
to the stochastic stability. When a complex system deviates from its
steady-state due to external perturbation or disfunctions of some
components, it is possible that the performance of system remains
normal. According to \cite{kitano2007towards, kitano2004biological}, such a property can
be evaluated by a performance function.

\medskip

\begin{dfn}
\label{perfun} The {\it performance function} $p(x)$ of system
\eqref{SDE1} is a continuous function on $\R^n$ such that
\begin{enumerate}
\item[a)] $p(x)=1, \quad \forall
x\in \mathcal{A}$;
\item[b)] $0<p(x)<1$, $x\notin \mathcal{A}$ \,.
\end{enumerate}
\end{dfn}

Following Kitano \cite{kitano2007towards}, one can define the {\em
functional $\epsilon$-robustness\index{Functional Robustness}}
${{R}_{f}}(\epsilon )$ w.r.t. $\sigma$ as
 \[{{R}_{f}}(\epsilon )=\int_{\R^n}{{{u }_{\epsilon }}}(x)p(x)\text{d}x,\]
where $u_{\epsilon}(x)$ is the stationary solution of \eqref{FPE1}.

\medskip

\begin{rem}
As $\epsilon \rightarrow 0$, $R_f(\epsilon)$ approaches to
$1$ for any continuous performance function. It is the rate of
convergence  of $R_{f}(\epsilon)$ to $1$ together with
the choice of the performance function that reveals the robustness
of system \eqref{SDE1}. For instance, if system \eqref{SDE1} has
strictly positive uniform robustness or 2-Wasserstein robustness,
the lower bound of functional robustness can be estimated.
\end{rem}



\medskip

\begin{pro}\label{pop55}
Assume $R_{w} > 0$ and $p(x)$ is twice differentiable, then there
exist positive constants $\epsilon_{0}$ and $C$ such that
\begin{displaymath}
   R_{f}(\epsilon) \geq 1-C \epsilon^{2}
\end{displaymath}
for all $\epsilon \in (0, \epsilon_{0})$.
\end{pro}

\begin{proof} It follow from  the definition of $R_{w}$ that
 there exists $\epsilon_{1} > 0$ such that
\begin{displaymath}
   \mathcal{W}^{2}(\mu_{\epsilon},\mu_{0}) < \frac{2\epsilon^{2}}{R_{w}^{2}}
\end{displaymath}
for all $0<\epsilon<\epsilon_{1}$. Hence by \eqref{WW},
\begin{displaymath}
   \int_{\R^n}
   \mathrm{dist}^{2}(x,\mathcal{A}) \mu_{\epsilon}(\mathrm{d}x) \leq
   \frac{2\epsilon^{2}}{R^{2}} : = V_{2}\epsilon^{2},
\end{displaymath}
for all $0<\epsilon<\epsilon_{1}$.

Since $p(x)$ is twice differentiable, there exists an open
neighborhood $\cal N$ of $\mathcal{A}$ and a positive constant $M$
such that $p(x) \geq 1 - M \mathrm{dist}^{2}(x,\mathcal{A})$ for all
$x \in \cal N$. Hence
\begin{displaymath}
   \int_{\R^n} u(x)p(x)\mathrm{d}x =
   \int_{\cal N} u(x)p(x)\mathrm{d}x + \int_{\mathbb{R}^{n}\backslash
  \cal  N} u(x)p(x)\mathrm{d}x := I_{1} + I_{2} \,.
\end{displaymath}

Let $d = \inf_{x \in \partial {\cal N}} \mathrm{dist}(x,
\mathcal{A})$. Then
$$
  1-\mu_{\epsilon}({\cal N}) = \int_{\mathbb{R}^{n}\setminus {\cal N}} \mathrm{d}\mu
  \leq
  \frac{1}{d^{2}}\int_{\mathbb{R}^{n}\setminus {\cal N}} \mathrm{dist}(x,
  \mathcal{A})^{2} \mathrm{d} \mu_{\epsilon} \leq
  \frac{V_{2}}{d^{2}}\epsilon^{2} \,.
$$

It follows that
\begin{eqnarray*}
I_{1} &\geq &\mu_{\epsilon}(\cal N) -
M\int_{\cal N} u(x)\mathrm{dist}^{2}(x,\mathcal{A})\mathrm{d}x \\
&=&1 - M\int_{\cal N}
u(x)\mathrm{dist}^{2}(x,\mathcal{A})\mathrm{d}x -
(1 - \mu_{\epsilon}( \cal N))\\
&\geq& 1- V_{2}M \epsilon^{2} - \frac{V_{2}}{d^{2}}\epsilon^{2} \,.
\end{eqnarray*}

Since $I_{2} \geq 0$, the proof is complete by letting $C = V_{2}M +
\frac{V_{2}}{d^{2}}$ and $\epsilon_{0} = \epsilon_{1}$.
\end{proof}
\medskip

\begin{pro}
Assume that {\bf H$^{0}$)} and {\bf H$^{1}$)} hold, $R_{u} > 0$ and $p(x)$ is twice
differentiable. Then there exist positive constants $\epsilon_{0}$,
$C$ such that
\begin{displaymath}
   R_{f}(\epsilon) \geq 1-C \epsilon^{2}
\end{displaymath}
for all $\epsilon \in (0, \epsilon_{0})$.
\end{pro}
\begin{proof}
It follows from Theorem ~\ref{MSD} that there exists
$\epsilon_{0}>0$ such that
$$
 \int_{\R^n} \mathrm{dist}^{2}(x, \mathcal{A}) \mu_{\epsilon}(\mathrm{d}x) \leq
  V_{2}\epsilon^{2}
$$
for all $\epsilon \in (0, \epsilon_{0})$. The rest of the proof is
identical to that of Proposition ~\ref{pop55}.
\end{proof}
\medskip

\begin{rem} We note that functional robustness does not imply
uniform robustness or 2-Wasserstein robustness. This is obvious by
letting $p(x) = 1$.
\end{rem}

\subsection{ Robustness of simple systems } In the case that $\cal A$ is a singleton, an explicit formula for the
2-Wasserstein robustness of \eqref{ODE1} w.r.t. any $\sigma$ can be
obtained.

\begin{pro}
\label{rfixpt} Assume that {\bf H$^{1}$)} holds and $\mathcal{A} =
\{x_{0}\}$.  If all eigenvalues of $Df(x_0)$ have negative real
parts, then
\begin{displaymath}
   R_{w} = \frac{\sqrt{2}}{\sqrt{{\bf
     Tr}(S^{-1})}}
\end{displaymath}
where $S$ solves the Lyapunov equation
\begin{displaymath}
   S (Df(x_0))^{\top} + Df(x_0) S^{\top} + A(x_{0}) = 0 \,.
\end{displaymath}
\end{pro}

\begin{proof}
According to the WKB expansion (see \cite{ludwig1975persistence,
day1985some}), there exists a quasi-potential function $V(x)$ and a
$C^{1}$ continuous function $w(x)$ with $w(x_{0}) = 1$ such that the
density function $u_{\epsilon}(x)$ of $\mu_{\epsilon}$ has the form
\begin{displaymath}
   u(x) = \frac{1}{K}e^{-V(x)/\epsilon^{2}}w(x) +
   o(\epsilon^{2}) \,.
\end{displaymath}
Moreover, it follows from \cite{day1985some} that $V(x)$ is of the
class $C^{3}$ in a neighborhood $N_{1}$ of $x_{0}$, and the Hessian
matrix of $V(x)$ at $x_{0}$ equals $S^{-1}/2$.  By
\cite{graham1982kronecker}, $S$ is a symmetric, positive definite
matrix.

Since $\mu_{\epsilon} \rightarrow \delta(x_{0})$ weakly, it
follows from Theorem ~\ref{Monge} that
\begin{displaymath}
   \mathcal{W}^{2}(\mu_{\epsilon},\delta (0)) =
   \int_{\mathbb{R}^{n}}|x-x_{0}|^{2}u_{\epsilon}(x)\mathrm{d}x \,.
\end{displaymath}

Denote $N = B(x_{0},\epsilon^{0.9})$ - the
$\epsilon^{0.9}$-neighborhood of $x_{0}$. Let $\epsilon_{0}>0$ be
small enough such that $N \subset \mathcal{N}\cap N_{1}$ for all
$0<\epsilon<\epsilon_{0}$, where $\cal N$ is as in {\bf H$^1$)}. Since
$w(x)$ is continuous,  we have $w(x) = 1 + O(\epsilon^{0.9})$, $x
\in N$, $0<\epsilon<\epsilon_{0}$.

Let $u$ be the density function of $\mu_{\epsilon} $ and
$$
  u_{0} = \frac{1}{K_{0}}e^{-(x-x_{0})^{\top}S^{-1}(x-x_{0})/2\epsilon^{2}} \,,
$$
where $K_{0}$ is the normalizer.

Then it is easy to check that the followings hold for all $x \in N$
and  $0<\epsilon<\epsilon_{0}$:
\begin{eqnarray*}
&&\frac{1}{\epsilon^{2}}|V(x) -
\frac{1}{2}(x-x_{0})^{\top}S(x-x_{0})|
\sim O(\epsilon^{0.7}) \,;\\
&&
w(x) = 1 + O(\epsilon^{0.9});\\
&& 1 - \mu_{\epsilon}(N) \sim o(\epsilon^{2});\\
 && \int_{\mathbb{R}^{n}\setminus N} u_{0}(x) \mathrm{d}x \sim
  o(\epsilon^{2}) \,.
\end{eqnarray*}
It follows from a straightforward calculation that $|\frac{K}{K_{0}}
- 1 | \sim O(\epsilon^{0.7})$. Thus,
$$
  |\frac{u_\epsilon(x)}{u_{0}(x)} - 1 | \sim O(\epsilon^{0.7}) \,
$$
for all $x \in N$, and consequently,
$$
  |\int_{N} |x-x_{0}|^{2}u_\epsilon(x) \mathrm{d}x - \int_{N} |x-x_{0}|^{2}u_{0}(x)
  \mathrm{d}x | \sim O(\epsilon^{2.5} )\,.
$$
Since
\begin{eqnarray*}
 && \int_{\mathbb{R}^{n}\setminus N} |x-x_{0}|^{2} u(x) \mathrm{d}x \sim
 o(\epsilon^{2}),\\
&&
  \int_{\mathbb{R}^{n}\setminus N} |x-x_{0}|^{2} u_{0}(x) \mathrm{d}x \sim
  o(\epsilon^{2}),
  \end{eqnarray*}
we have
$$
  \int_{\mathbb{R}^{n}}|x-x_{0}|^{2}u_{\epsilon}(x)\mathrm{d}x  = \int\frac{1}{K_{0}}|x|^{2}e^{-x^{\top}S^{-1}x/2\epsilon^{2}}
  \mathrm{d}x + o(\epsilon^{2}) \,
$$
for any $\epsilon \in (0, \epsilon_{0})$. The rest of the proof
follows from the definition of $R_w$ and direct  calculations.
\end{proof}

%

\section{Connections among Degeneracy, complexity and  robustness}
\label{connections}

It has been observed in neural systems that a higher degeneracy is
always accompanied by a high complexity \cite{tononi1999measures,
edelman2001degeneracy, whitacre2012biological, clark2011degeneracy}. We will show in this section that this is
also the case for
 a biological network described by ODE system with respect to a fixed noise matrix $\sigma$.

Unlike the connections between  degeneracy and complexity,
robustness of system \eqref{ODE1} alone does not necessarily imply
its degeneracy or complexity with respect to a given noise
perturbation $\sigma$. As a simple example, the completely decoupled
linear system $x_{i}' = -x_{i}$, $i=1,2,\cdots, n$, has zero
complexity hence zero degeneracy with respect to $\sigma(x)\equiv
Id$ according to Theorem~\ref{deglarger}, but it is uniformly
robust. In this section, we will exam two special cases of
\eqref{ODE1} under either geometric or dynamical condition of its
global attractor $\cal A$ for which degeneracy is actually accompanied
by high robustness. This agrees with the cases of neural systems that
robustness can arise from a variety of sources; while degeneracy is only
one of these sources \cite{whitacre2012biological}.

\subsection{Degeneracy implies Complexity} Through this subsection,
we let $\sigma$ be a fixed noise matrix.

\begin{lem}
\label{deg2com} With respect to any probability density function on
$\mathbb{R}^{n}$ and a given decomposition
$\mathbb{R}^{n} = I_{k}\oplus I_{k}^{c} \oplus \mathcal{O}$, we have
\begin{equation}
   \label{degcom}
   MI(I_{k}; I_{k}^{c};\mathcal{O})\le \min
   \{MI({{I}_{k}}; I_{k}^{c}),MI(I_{k}^{c};\mathcal{O}),MI(I_{k};\mathcal{O})\}.
\end{equation}
\end{lem}
\begin{proof}

It is sufficient to prove that for any three random variables $X, Y, Z$ with
joint probability density function $P(x,y,z)$,
$$
  MI(X; Y; Z) \leq \min \{ MI(X; Y), MI(Y; Z), MI(X; Z)\} \,.
$$

It follows from the definition of mutual information that
\begin{eqnarray*}
MI(X;Y;Z) &= &H(X)+H(Y)+H(Z)-H(X,Y)-H(Y,Z)\\
& & -H(X,Z)+H(X,Y,Z) \\
&=&H(X)+H(Y)-H(X,Y) \\
& & - ( H(X,Z)+H(Y,Z)-H(Z)-H(X,Y,Z))\\
&=&MI(X;Y) - MI(X;Y \, | \,Z) \,,
\end{eqnarray*}
where the latter term $MI(X;Y \, | \, Z)$ is the conditional mutual
information. Thus it is sufficient to prove that $ MI(X;Y|Z)\geq 0$.

\medskip

The nonnegativity of conditional mutual information is a direct
corollary of Kullback's inequality \cite{kullback1997information}.
For the sake of completeness, we borrow the following proof from
\cite{yeung2002first}. Let $P(x,y,z)$ be the joint probability
density function. The marginal probability density functions and
conditional probability functions are denoted by $P(x), P(y), \cdots
$ and $P(x,y \,|\, z), P(x \,|\, y,z) , \cdots$ respectively.
Then
\begin{eqnarray*}
MI(X;Y|Z) &=& \int P(x,y,z) \left [ \log P(x,y,z) + \log P(z) - \log
  P(x,z) \right .
\\
&& - \left .
\log P(y,z) \right ] \mathrm{d}x \mathrm{d}y \mathrm{d}z \\
&=& \int P(x,y,z) \log \left \{ \frac{ P(x,y,z) / P(z)}{ P(x,z)/P(z)
    \cdot P(y,z)/P(z)} \right \} \mathrm{d}x \mathrm{d}y \mathrm{d}z \\
&=&\int
P(x,y,z)\log\frac{P(x,y|z)}{P(x|z)P(y|z)}\mathrm{d}x\mathrm{d}y\mathrm{d}z
\\
&=&\int P(z)\left \{\int
P(x,y|z)\log\frac{P(x,y|z)}{P(x|z)P(y|z)}\mathrm{d}x\mathrm{d}y
\right \}\mathrm{d}z \,.
\end{eqnarray*}

>From Kullback's inequality \cite{kullback1997information}, for any $z$ there holds
\begin{displaymath}
   \int
P(x,y|z)\log\frac{P(x,y|z)}{P(x|z)P(y|z)}\mathrm{d}x\mathrm{d}y \geq
0 \,.
\end{displaymath}

\medskip

Inequalities $MI(X;Y;Z)\leq MI(X;Z)$ and $MI(X;Y;Z)\leq
MI(Y;Z)$ can be proved analogously. This leads to the inequality
\eqref{degcom}.
\end{proof}

\begin{thm}
\label{deglarger}
The complexity of a system is no less than its degeneracy.
\end{thm}
\begin{proof}

Fix $\epsilon>0$ and noise matrix $\sigma$. Let $\mathcal{O}$ be the
coordinate subspace of $\mathbb{R}^{n}$ as before. Let $\{ I_{k},
I_{k}^{c}, \mathcal{O} \} $ be any decomposition of coordinate subspaces as
described in Section 3.1. Then by Lemma ~\ref{deg2com},
$$
  MI(I_{k}; I_{k}^{c};\mathcal{O})\le MI({{I}_{k}}; I_{k}^{c}) \,.
$$
Since mutual information $MI({{I}_{k}}; I_{k}^{c}) $ is nonnegative,
$\max\{MI(I,I^{c}_{k};\mathcal{O}), 0 \} \leq
MI(I;I^{c}_{k})$. Comparing equation \eqref{degeneracy2} with
\eqref{complexity}, one obtains
    \[\mathcal{C}(\mathcal{O})\ge \mathcal{D}(\mathcal{O}).\]


\medskip

By taking the supreme over all the subspace $\mathcal{O}$, it is
easy to see that $\mathcal{D}_{\epsilon,\sigma} \leq
\mathcal{C}_{\epsilon,
  \sigma}$. The proof is completed by taking the limit infimum
over $\epsilon > 0$ and taking the supremum over $\sigma$ with respect to unit norm.
\end{proof}

\subsection{Robust systems with non-degenerate global attractor}
For a system to have positive degeneracy, the system must be complex.
Geometrically such structural complexity often gives rise to some
kind of embedding complexity of the global attractor into the phase
space. Roughly speaking, the components of a complex system interact
strongly with one another and as a result, the global attractor is
non-degenerate in the phase space such that it does not lay in any
coordinate subspace. To characterize the non-degenerate property of the global
attractor, it is natural to consider its projections on certain
coordinate subspace and measure the dimensions of the corresponding
projections. We note that the attractor as well as its projections
may only be fractal sets, hence they should be measured with respect
to the Minkowski dimension, also called box counting dimension
\cite{pesin1997dimension}.

For any coordinate  subspace $\mathcal{V}$ of ${{R}^{n}}$, we denote
by ${{d}_{\mathcal{V}}}$ the co-dimension of $\mathcal{A}$ in
$\mathcal{V}$, i.e., the dimension of $\mathcal{V}$ minus the
Minkowski dimension of the projection of $\mathcal{A}$ to
$\mathcal{V}$.

\medskip
\begin{dfn}
The global attractor $\mathcal{A}$ is said to be {\em non-degenerate} if
$\mathcal{A}$ is a regular set and
there is a coordinate  decomposition ${{\mathbb{R}}^{n}}=I \oplus J
\oplus \mathcal{O}$ such that

\[d_{I}+d_{J}+d_{\mathcal{O}}+d_{\mathbb{R}^{n}}<
d_{I\oplus J}+d_{I \oplus
            \mathcal{O}}+d_{J\oplus \mathcal{O}}.\]
\end{dfn}
\medskip

A sufficient condition for a set to be non-degenerate is that the
dimension of the set does not decrease after projecting it onto
coordinate subspaces. The following proposition follows from some
straightforward calculation.

\begin{pro}
Let $P_{\mathcal{V}}$ be the projection operator onto a subspace $\mathcal{V}$ of
$\mathbb{R}^{n}$. If a regular set $\mathcal{A}$ with strictly
positive dimension satisfies $\mathrm{dim}(P_{\mathcal{V}}
\mathcal{A}) = \mathrm{dim}(\mathcal{A})$ for $\mathcal{V} = I, J$,
and $\mathcal{O}$, then $\mathcal{A}$ is degenerate.
\end{pro}
\begin{proof}
Since all projections do not change the dimension of $\mathcal{A}$, we
have
\begin{eqnarray*}
&& d_{I}+d_{J}+d_{\mathcal{O}}+d_{\mathbb{R}^{n}} \\
&=& (\mathrm{dim}(I) -
  \mathrm{dim} (A) )+ (\mathrm{dim}(J) -
  \mathrm{dim} (A) ) + (\mathrm{dim}(\mathcal{O}) -
  \mathrm{dim} (A) ) + n - \mathrm{dim}(A) \\
&<&  (\mathrm{dim}(I) -
  \mathrm{dim} (A) )+ (\mathrm{dim}(J) -
  \mathrm{dim} (A) ) + (\mathrm{dim}(\mathcal{O}) -
  \mathrm{dim} (A) ) \\
&& + \mathrm{dim}(I) + \mathrm{dim}(J) + \mathrm{dim}(\mathcal{O})\\
&=&  d_{I\oplus J}+d_{I \oplus
            \mathcal{O}}+d_{J\oplus \mathcal{O}} \,.
\end{eqnarray*}

\end{proof}

The following theorem says that geometric complexity of the global
attractor of a system can imply its degeneracy.


\begin{thm} {\rm (Non-degenerate Attractor)}
\label{twisted} Assume that both {\bf H$^{0}$)} and {\bf H$^{1}$)}
hold. If the global attractor $\mathcal{A}$ is non-degenerate  and each
$\mu_{\epsilon}$ is regular with respect to $\mathcal{A}$, then
there exists an $\epsilon_{0}>0$, such that
$\mathcal{D}_{\epsilon,\sigma}>0$ for all $\epsilon \in (0, \epsilon_{0})$.
\end{thm}

\begin{proof}
Since each $\mu_\epsilon$ is regular, we have by Theorem
~\ref{EntDimThm} that
\begin{equation*}
   \lim_{\epsilon\rightarrow
     0}\frac{ \mathcal{H} (\mu_{\epsilon})}{\log \epsilon} = n-d \,.
\end{equation*}

Let $I$ be a coordinate  subspace of $\mathbb{R}^{n}$ and $P$ be the
projection operator onto $I$. For simplicity, we suspend the
$\epsilon$-dependency and let $u(x)$ be the density function of
$\mu_\epsilon$ for fixed $\epsilon$. Denote $u_{I} = Pu$ as the
marginal distribution of $u(x)$ on $I$. We first show that all
marginal distribution $u_{I}$ satisfy the entropy-dimension
identity.

For a fixed $\delta>0$, it follows from the definition of a regular
 invariant measure with respect to $\mathcal{A}$ that there
exist $K<\infty$, $\epsilon_{1}>0$ and a family of approximate
functions $u_{K, \epsilon}$ supported on $B( \mathcal{A}, K\epsilon)$
such that for all $\epsilon \in (0, \epsilon_{1})$, the $L^{1}$ error
between $u_{K, \epsilon}$ and $u$ is smaller than $\delta$.

Let $u_{2} = u-u_{K,\epsilon}$, $\bar{u}_{1} = Pu_{K, \epsilon}$ and
$\bar{u}_{2} = Pu_{2}$. Then the projected entropy on $I$ satisfies
$$
  \int_{I} u_{I}(x) \log u_{I}(x)\mathrm{d}x = \int_{I}(\bar{u}_{1}
    (x_{1})+\bar{u}_{2}(x_{1})) \log(\bar {u}_{1}(x_{1}) + \bar{u}_{2} (x_{1}))
  \mathrm{d}x_{1}.
$$

Therefore,
\begin{eqnarray*}
  H(I)= H(Pu)  & = & \int_{I}(\bar{u}_{1}(x_{1})+\bar{u}_{2}(x_{1})) \log(\bar{u}_{1}(x_{1}) + \bar{u}_{2}(x_{1}))
  \mathrm{d}x_{1}\\
& = &  \int_{I}(\bar{u}_{1}(x_{1})+\bar{u}_{2}(x_{1})) \left [ \log \bar{u}_{1}(x_{1}) +
  \log (1 + \frac{ \bar{u}_{2}(x_{1})}{\bar{u}_{1}(x_{1})} ) \right ]
\mathrm{d}x_{1} \\
&\geq& \int_{I}(\bar{u}_{1}(x_{1})+\bar{u}_{2}(x_{1})) \left [ \log \bar{u}_{1}(x_{1}) +
  \frac{ \bar{u}_{2}(x_{1})/ \bar{u}_{1}(x_{1})}{1 + \bar{u}_{2}(x_{1})/ \bar{u}_{1}(x_{1})}
\right ] \mathrm{d}x_{1} \\
&\geq& \int_{I} \bar{u}_{1}(x_{1}) \log \bar{u}_{1} (x_{1})\mathrm{d}x_{1} - \int_{I} |
\bar{u}_{2}(x_{1})|(1 +  |\log \bar{u}_{1}(x_{1})|) \mathrm{d}x_{1} := I_{1} - I_{2} \,.
\end{eqnarray*}
Furthermore, it follows from the convexity of $x\log x$ that
\begin{eqnarray*}
&&H(I)= H(Pu) \\
&= &  \int_{I}(\bar{u}_{1}(x_{1})+\bar{u}_{2}(x_{1})) \log(\bar{u}_{1}(x_{1}) + \bar{u}_{2}(x_{1}))
  \mathrm{d}x_{1}\\
&\leq& \int_{I}(\bar{u}_{1}(x_{1})+|\bar{u}_{2}(x_{1})|) \log(\bar{u}_{1}(x_{1}) + |\bar{u}_{2}(x_{1})|)
  \mathrm{d}x_{1} \\
&&+ 2 \int_{I} | \bar{u}_{2}(x_{1})| |\log( \bar{u}_{1}(x_{1}) + |
  \bar{u}_{2}(x_{1})|) | \mathrm{d}x_{1}\\
&\leq& 2\int_{I}\frac{\bar{u}_{1}(x_{1})+|\bar{u}_{2}(x_{1})|}{2}
\log(\frac{\bar{u}_{1}(x_{1}) + |\bar{u}_{2}
    (x_{1})|}{2}) \mathrm{d}x_{1} \\
&& +2 \int_{I} | \bar{u}_{2}(x_{1})| |\log( \bar{u}_{1}(x_{1}) + |
  \bar{u}_{2}(x_{1})|)| \mathrm{d}x_{1}\\
&& + \log 2\\
&\leq& \int_{I} \bar{u}_{1}(x_{1})\log \bar{u}_{1}(x_{1})
\mathrm{d}x_{1} \\
&&+ \int_{I}
|\bar{u}_{2}(x_{1})|\left[ \log |\bar{u}_{2}(x_{1})| + |\log( \bar{u}_{1}(x_{1}) + |
  \bar{u}_{2}(x_{1})|)|  \right ]\mathrm{d}x_{1} + \log 2\\
&:=& I_{1} + I_{2} + \log 2 \,.
\end{eqnarray*}

To estimate $I_1$, we note from Section 2.3 the definitions
of regular set and stationary measure that there are constants
$C_{1},C_{2}$ independent of $\epsilon$ such that
$$
  (1-\delta){ d_I} (- \log \epsilon) -C_{1} \leq I_{1}\leq { d_I} (-\log \epsilon) +
  C_{2}.
$$

To estimate $I_2$, we note that
$$
\int_{I} |\bar{u}_{2}|(x) \mathrm{d}x = \int_{\mathbb{R}^{n}} |u_{2}|(x) \mathrm{d}x < \delta
$$
and from Lemma 3.3  that $|\bar{u}_{2}(x)| <
\epsilon^{-(2n+2)}$. Thus $I_{2} \leq (2n+2) \delta(-\log
\epsilon)$. Similarly $I_{3} \leq (4n+4) \delta(-\log \epsilon)$. Summarizing the above, we have
$$
  (1-\delta) d_I\leq \lim_{\epsilon\rightarrow 0}
  \frac{H(I)}{-\log \epsilon} \leq (1+3 (2n+2)\delta)d_I \,.
$$
As the above inequality holds for any $\delta>0$, we have
\begin{equation}\label{ent-dim}
 \lim_{\epsilon\rightarrow 0}
  \frac{H(I)}{-\log \epsilon} = d_I.
\end{equation}

Let $\mathbb{R}^{n} = I \oplus J \oplus \mathcal{O}$ be a coordinate
decomposition such that
$$
  d_{I} + d_{J} + d_{\mathcal{O}} + d_{\mathbb{R}^{n}} < d_{I\oplus
  J} + d_{I \oplus \mathcal{O}} + d_{J \oplus \mathcal{O}} \,.
$$
Since
$$
  MI(I;J; \mathcal{O}) = H(I) + H(J) +
  H(O) + H(\mathbb{R}^{n}) - H(I
  \oplus J) - H(I\oplus \mathcal{O}) - H(J
  \oplus \mathcal{O}),
$$
applications of \eqref{ent-dim} to $I,J,\cal O,I\oplus
  J, I \oplus \mathcal{O},J \oplus \mathcal{O} $, respectively,
yield that
$$
  MI(I;J; \mathcal{O})  \simeq (d_{I} + d_{J} + d_{\mathcal{O}} + d_{\mathbb{R}^{n}} - d_{I\oplus
  J} - d_{I \oplus \mathcal{O}} - d_{J \oplus \mathcal{O}} )\log
\epsilon > 0,
$$
from which the theorem follows.
\end{proof}

\medskip




\begin{ex}
Consider the system
\begin{equation}
  \label{limitcycle}
 \left\{\begin{array}{l}
  x' = y + x(1 - x^{2} - y^{2}) + \epsilon \mathrm{d}W_{t}\\
 y' = -x + y(1 - x^{2} - y^{2}) + \epsilon \mathrm{d}W_{t}\\
  z' =  - z + \epsilon \mathrm{d}W_{t}
\end{array}\right.
\end{equation}
It is easy to verify that
$$
  v(x,y,z) = \frac{1}{Z}\exp \{ -\epsilon^{-2} ( \frac{1}{2}z^{2} + \frac{1}{4}
  (1 - x^{2} - y^{2})^{2}\}
$$
is a stationary density function of \eqref{limitcycle}, where $Z$ is
the normalizer. Therefore assumption {\bf H$^{1}$)} is satisfied and
function $v(x,y,z)$ is regular with respect to
$\mathcal{A} = \{ (x,y,z) : \, x^{2} + y^{2} = 1 \}$. However,
$\mathcal{A}$ is not a non-degenerate attractor because $\mathcal{A}$
lies on the plane $z = 0$.

If we change coordinates such that $\mathcal{A}$ is not contained
in any coordinate subspace, e.g. \'  via
coordinate change $(x,y,z) = (u,v,u+v+w)$, then  under
the new coordinate $\mathcal{A}$ becomes a non-degenerate attractor
and Theorem ~\ref{twisted} is applicable to system
\eqref{limitcycle}.


\end{ex}

\subsection{Simple robust systems}
Degenerate phenomenon can also occur when the attractor $\cal A$ of
system \eqref{ODE1} is both geometrically and dynamically simple.
Below, we exam the case of a simple system in which the global
attractor $\cal A$ is an  exponentially attracting
equilibrium - a so-called homeostatic system in  biological term.
 We note that  such a system automatically satisfy the
condition {\bf H$^0$)}, hence it is  robust according to
Propositions~\ref{prop4.1}. We will show that if in a neighborhood
of the globally attracting equilibrium different directions
demonstrate different sensitivities with respect to the noise
perturbation, then the system must be degenerate.

Let $S=(s_{ij})$ be an $n\times n$ matrix and $I$ be a coordinate
subspace of $\R^n$ spanned by standard unit vectors
$\{e_{i_1},\cdots,e_{i_k}\}$ for some $k\le n$. Denote
$S(I)=(a_{i_li_m})_{1\le l,m\le k}$  and $|S(I)|$ the determinant of
$S(I)$.

\begin{thm}\label{degfixpt} {\rm (Degeneracy of simple systems)}
Assume that {\bf H$^{1}$)}  holds, $\mathcal{A}$ is an equilibrium
$\{x_0\}$,  and all eigenvalues of $Df(x_0)$ have
negative real parts. Then the following holds:
\begin{itemize}
\item[{\rm a)}] With respect to any coordinate decomposition
$\R^n=I_1\oplus I_2\oplus \mathcal{O}$,
\begin{equation}
\label{degfix} \lim_{\epsilon\rightarrow 0} MI(I_{1};I_{2};\mathcal{O})  =
\frac{1}{2}  \log
\frac{|S({{I}_{1}})||S({{I}_{2}})||S(\mathcal{O})||S(I_{1}\oplus I_{2}\oplus
\mathcal{O})|}{|S({{I}_{1}}\oplus{{I}_{2}})||S({{I}_ {1}}\oplus
\mathcal{O})||S({{I}_{2}} \oplus \mathcal{O})|} \,,
\end{equation}
where $S$ solves equation
$$
     S{{J}^{\top}}+JS+A(x_{0})=0 \,.
$$
Consequently, if, with respect to a given coordinate decomposition
$\R^{n} = I_{1} \oplus I_{2} \oplus \mathcal{O}$,
\begin{equation}
\label{eq610}
  \log
\frac{|S({{I}_{1}})||S({{I}_{2}})||S(\mathcal{O})||S(I_{1}\oplus I_{2}\oplus
\mathcal{O})|}{|S({{I}_{1}}\oplus{{I}_{2}})||S({{I}_ {1}}\oplus
\mathcal{O})||S({{I}_{2}} \oplus \mathcal{O})|} >0 \,,
\end{equation}
then the $\sigma$-degeneracy of system \eqref{ODE1} is positive.

\item[{\rm b)}] The $\sigma$-degeneracy of \eqref{ODE1} continuously depends on
$Df(x_0)$.
\end{itemize}
\end{thm}
\begin{proof}
For simplicity, denote $J=Df(x_0)$, $A = A(x_{0})$, and $u(x)$ as
the density function of $\mu_\epsilon$.

a) By \cite{day1985some,
  ludwig1975persistence, ery1993noise}, $u(x)$ admits the following WKB expansion
\begin{equation}
\label{WKB1}
   u(x) = \frac{1}{K}e^{-V(x)/\epsilon^{2}}w(x) + o(\epsilon^{2})
\end{equation}
for some quasipotential function $V(x)$ and some $C^{1}$  function
$w(x)$ with $w(x_{0}) = 1$. Moreover,  $V(x)$ is twice
differentiable in an open neighborhood $N(x_{0})$ of $x_0$ and it
can be approximated by $x^{\top}S^{-1} x/2$, where  $S$ is the
positive definite matrix uniquely solving the Lyapunov equation
\begin{equation}
   \label{Lyeqn}
   S{{J}^{\top}}+JS+A=0 \,.
\end{equation}

  Let $\nu_{\epsilon}$ be the Gibbs measure with density function
\begin{equation}
\label{gibbs}
  u_{0} (x) = \frac{1}{K_{0}}e^{-x^{\top}S^{-1}x/2\epsilon^{2}} \,,
\end{equation}
where $K_{0}$ is the normalizer. Obviously $u_{0}$ is a multivariate
 with covariance matrix $\epsilon^{2}S$. The margin of $u_{0}$ on
any coordinate subspace  $I$ has covariance matrix
$\epsilon^{2}S(I)$. Recall that the entropy of a $k$-variable normal
distribution with covariance matrix $\Sigma$ reads $
  \frac{1}{2}\log ((2\pi e)^{k}|\Sigma|).
$ Using this fact, simple calculations  show that, with respect to
any coordinate decomposition $\R^n=I_1\oplus I_2\oplus \mathcal{O}$, the
multivariate mutual information $MI_0(I_{1}; I_{2}; \mathcal{O})$ of $u_{0}$ satisfies
$$
  \lim_{\epsilon\to 0} MI_0(I_{1}; I_{2}; \mathcal{O}) =  \frac{1}{2}  \log
\frac{|S({{I}_{1}})||S({{I}_{2}})||S(\mathcal{O})||S(I_{1}\oplus I_{2}\oplus
\mathcal{O})|}{|S({{I}_{1}}\oplus{{I}_{2}})||S({{I}_ {1}}\oplus
\mathcal{O})||S({{I}_{2}} \oplus \mathcal{O})|}.
$$

The proof of \eqref{degfix} amounts to  show that
\begin{equation}\label{degfix1}
\lim_{\epsilon\to 0}| MI(I_{1}; I_{2}; \mathcal{O})- MI_0(I_{1}; I_{2}; \mathcal{O})|=0.
\end{equation}
We first show that
\begin{equation}
\label{pg111}
   \lim_{\epsilon\rightarrow 0}|H(\mu_{\epsilon}) -
   H(\nu_{\epsilon})| = 0 \,.
\end{equation}
Without loss of generality, we assume that the isolating neighborhood
$\mathcal{N}$ in  {\bf H$^1$)}   satisfies $\mathcal{N} \subseteq
N(x_{0})$. Let $\Delta_{\epsilon} = \{x |\|x - x_{0}\| \leq
\epsilon^{4/5}\} $. We will prove \eqref{pg111} in two steps.
\medskip

\noindent{\em Claim 1:} $\displaystyle
 \lim_{\epsilon\rightarrow 0} \int_{\mathbb{R}^{n}\setminus\Delta_{\epsilon}} u(x) \log u(x)
  \mathrm{d}x =\lim_{\epsilon\rightarrow 0} \int_{\mathbb{R}^{n}\setminus\Delta_{\epsilon}} u_{0}(x) \log u_{0}(x)
  \mathrm{d}x =0.
$
\medskip

On one hand, since both $u_{0}(x)$ and $u(x)$ satisfy {\bf H$^{1}$)},
by Lemma \ref{entbound} we
have
$$
  u_{0}(x) < \epsilon^{-(2n+1)}, \quad u(x) <
  \epsilon^{-(2n+1)},\quad \epsilon\ll 1,
$$
and
$$
  \int_{\mathbb{R}^{n}\setminus\Delta_{\epsilon}} u(x) \mathrm{d}x
  \sim o(\epsilon^{2}).
$$
It is also clear that
$$
  \int_{\mathbb{R}^{n}\setminus\Delta_{\epsilon}} u_{0}(x) \mathrm{d}x
  \sim o(\epsilon^{2}).
$$
It follows that
$$
  \lim_{\epsilon\rightarrow 0} \int_{\mathbb{R}^{n}\setminus\Delta_{\epsilon}} u(x) \log u(x)
  \mathrm{d}x \leq \lim_{\epsilon\rightarrow 0} \epsilon^{2}\log
  \epsilon = 0
$$
and
$$
  \lim_{\epsilon\rightarrow 0} \int_{\mathbb{R}^{n}\setminus\Delta_{\epsilon}} u_{0}(x) \log u_{0}(x)
  \mathrm{d}x \leq \lim_{\epsilon\rightarrow 0} \epsilon^{2}\log
  \epsilon = 0 \,.
$$

On the other hand, we have by Lemmas ~\ref{entout},~ \ref{entin} that
there is a  constant $R_{0}
> 0 $ such that
\begin{eqnarray*}
&&  \int_{\mathbb{R}^{n}\setminus\Delta_{\epsilon}} u(x) \log u(x)
  \mathrm{d}x\\
&=& \int_{\mathbb{R}^{n}\setminus B(0, R_{0})} u(x) \log u(x)
  \mathrm{d}x +
  \int_{B(0,R_{0})\setminus \Delta_{\epsilon}} u(x) \log u(x)
  \mathrm{d}x \geq -\epsilon^{2} - 2 \sqrt{\epsilon},\\
&& \int_{\mathbb{R}^{n}\setminus\Delta_{\epsilon}} u_{0}(x) \log
u_{0}(x)
  \mathrm{d}x\\
& =&\int_{\mathbb{R}^{n}\setminus B(0, R_{0})} u_{0}(x) \log u_{0}(x)
  \mathrm{d}x +
  \int_{B(0,R_{0})\setminus \Delta_{\epsilon}} u_{0}(x) \log u_{0}(x)
  \mathrm{d}x\geq -\epsilon^{2} - 2 \sqrt{\epsilon},
\end{eqnarray*}
whenever $\epsilon$ is sufficiently small. Hence
$$
 \lim_{\epsilon\rightarrow 0} \int_{\mathbb{R}^{n}\setminus\Delta_{\epsilon}} u(x) \log u(x)
  \mathrm{d}x  \geq 0,\;\,
 \lim_{\epsilon\rightarrow 0} \int_{\mathbb{R}^{n}\setminus\Delta_{\epsilon}} u_{0}(x) \log u_{0}(x)
  \mathrm{d}x  \geq 0.
$$
This proves Claim 1.
\medskip

\noindent{\em Claim 2}: $\displaystyle\lim_{\epsilon\to 0}
  |\int_{\Delta_{\epsilon}} u(x) \log u(x) \mathrm{d}x -
  \int_{\Delta_{\epsilon}} u_{0}(x)\log u_{0}(x) \mathrm{d}x|=0. $

\medskip

We note that
\begin{displaymath}
   K =
   \frac{1}{\mu_{\epsilon}(\Delta_{\epsilon})}\int_{\Delta_{\epsilon}}
   e^{-V(x)/\epsilon^{2}}z(x)\mathrm{d}x,\;\, K_{0} = \frac{1}{\nu_{\epsilon}(\Delta_{\epsilon})}\int
   e^{-x^{\top}Sx/\epsilon}\mathrm{d}x\,.
\end{displaymath}

It is easy to check that
\begin{eqnarray}
&&\frac{1}{\epsilon^{2}}|V(x) -
\frac{1}{2}(x-x_{0})^{\top}S(x-x_{0})| \sim O(\epsilon^{2/5}),\;\; x
\in \Delta_{\epsilon};\label{(i)}\\
&& w(x) = 1 + O(\epsilon^{0.8}),\;\; x \in {\cal N};\label{(ii)}\\
&& 1 - \mu_{\epsilon}(\Delta_{\epsilon}) \sim
o(\epsilon^{2});\label{(iii)}\\
 && \int_{\mathbb{R}^{n}\setminus \Delta_{\epsilon}} u_{0}(x) \mathrm{d}x \sim
  o(\epsilon^{2}) \,.\label{(iv)}
\end{eqnarray}
It follows from straightforward calculations using
\eqref{(i)}-\eqref{(iv)} that $|\frac{K}{K_{0}} - 1 | \sim
O(\epsilon^{2/5})$. Thus,
$$
  |\frac{u(x)}{u_{0}(x)} - 1 | \sim O(\epsilon^{2/5}),\;\;\, x\in \cal
  N,
$$
and consequently,
\begin{eqnarray*}
&&|\int_{\Delta_{\epsilon}} u(x) \log u(x) \mathrm{d}x
-\int_{\Delta_{\epsilon}} u(x) \log u(x) \mathrm{d}x |\\
&\leq& \int_{\Delta_{\epsilon}} u(x) |\log(\frac{u(x)}{u_{0}(x)}|
\mathrm{d}x + \int_{\Delta_{\epsilon}} |u_{0}(x)\log u_{0}(x)
(\frac{u(x)}{u_{0}(x)} - 1) | \mathrm{d}x\\
&=& O(\epsilon^{2/5}) + O( \epsilon^{2/5}\log \epsilon).
\end{eqnarray*}
This proves Claim 2. \eqref{pg111} now follows from the above two
claims.

\medskip

Next, we show that with respect to any coordinate subspace the
projected entropy of  $u_{0}$ is still an approximation of that of
$u$.

Let $x = (x_{1}, x_{2})$ be a decomposition of coordinates of
$\mathbb{R}^{n}$ and let $\bar{u}(x_{1})$ and $\bar{u}_{0}(x_{1})$
be the projection of $u$ and $u_{0}$ respectively such that $x_{1} \in
\mathbb{R}^{m}$. Denote
$\bar{\Delta}_{\epsilon} = \{x_{1}:\, |x_{1}| < \epsilon^{4/5} \}$.
Then the same proof as that for Claim 1 yields that
\begin{equation}
  \label{5-4-1}
\lim_{\epsilon\rightarrow 0} \int_{\mathbb{R}^{m}\setminus\bar{\Delta}_{\epsilon}} \bar{u}(x_{1}) \log \bar{u}(x_{1})
  \mathrm{d}x_{1} =
  \lim_{\epsilon\rightarrow 0} \int_{\mathbb{R}^{m}\setminus\bar{\Delta}_{\epsilon}} \bar{u}_{0}(x_{1}) \log \bar{u}_{0}(x_{1})
  \mathrm{d}x_{1} =0 \,.
\end{equation}

Denote
$$
  \hat{u}(x_{1}) = \int_{\{|x_{2}| \leq \epsilon^{4/5}\}} u(x_{1}, x_{2})
  \mathrm{d}x_{2},\;\;
  \hat{u}_{0}(x_{1}) = \int_{\{|x_{2}| \leq \epsilon^{4/5}\}} u_{0}(x_{1}, x_{2})
  \mathrm{d}x_{2} \,.
$$
Similar to the proof of Claim 2, we have
\begin{equation}
  \label{5-4-2}
\lim_{\epsilon\rightarrow 0}  |\int_{\bar{\Delta}_{\epsilon}}
\hat{u}(x_{1})\log \hat{u}(x_{1}) \mathrm{d}x_{1} - \int_{\bar{\Delta}_{\epsilon}}
\hat{u}_{0}(x_{1})\log
  \hat{u}_{0}(x_{1}) \mathrm{d}x_{1} | = 0.
\end{equation}

Note that
\begin{eqnarray*}
 & &| \int_{\mathbb{R}^{m}} \bar{u}(x_{1}) \log \bar{u}(x_{1}) \mathrm{d}x_{1} - \int_{\mathbb{R}^{m}} \bar{u}_{0}(x_{1})
 \log \bar{u}_{0}(x_{1}) \mathrm{d}x_{1} | \\
&\leq&| \int_{\mathbb{R}^{m}\setminus\bar{\Delta}_{\epsilon}} \bar{u}(x_{1})
\log \bar{u}(x_{1}) \mathrm{d}x_{1}| + |
\int_{\mathbb{R}^{m}\setminus\bar{\Delta}_{\epsilon}} \bar{u}_{0}(x_{1})\log
\bar{u}_{0}(x_{1}) \mathrm{d}x_{1} | \\
&& +  |\int_{\bar{\Delta}_{\epsilon}}
\hat{u}(x_{1})\log \hat{u}
  (x_{1})\mathrm{d}x_{1} - \int_{\bar{\Delta}_{\epsilon}}
\hat{u}_{0}(x_{1})\log
  \hat{u}_{0} (x_{1})\mathrm{d}x_{1} | \\
&& + |\int_{\bar{\Delta}_{\epsilon}} \hat{u}(x_{1})\log \hat{u}(x_{1}) \mathrm{d}x_{1} - \int_{\bar{\Delta}_{\epsilon}} \bar{u}(x_{1})\log
  \bar{u}(x_{1}) \mathrm{d}x_{1} | \\
&&+ |\int_{\bar{\Delta}_{\epsilon}}
  \hat{u}_{0}(x_{1})\log \hat{u}_{0}(x_{1}) \mathrm{d}x_{1} -
  \int_{\bar{\Delta}_{\epsilon} } \bar{u}_{0}(x_{1})\log
  \bar{u}_{0} (x_{1})\mathrm{d}x_{1} | \,.
\end{eqnarray*}

By equations \eqref{5-4-1} and \eqref{5-4-2}, it is sufficient to
show that as $\epsilon \rightarrow 0$,
$$
  \lim_{\epsilon\rightarrow 0} |\int_{\bar{\Delta}_{\epsilon}} \hat{u}(x_{1})\log \hat{u}(x_{1}) \mathrm{d}x_{1} - \int_{\bar{\Delta}_{\epsilon}} \bar{u}(x_{1})\log
  \bar{u} (x_{1})\mathrm{d}x_{1} | = 0,
$$
and
$$
  \lim_{\epsilon\rightarrow 0} |\int_{\bar{\Delta}_{\epsilon}} \hat{u}_{0}(x_{1})\log \hat{u}_{0}(x_{1}) \mathrm{d}x_{1} - \int_{\bar{\Delta}_{\epsilon}} \bar{u}_{0}(x_{1})\log
  \bar{u}_{0} (x_{1})\mathrm{d}x_{1} | = 0.
$$

The convergence with respect to $\bar{u}_{0}$ and $\hat{u}_{0}$
follows directly from the expression of $u_{0}$. For the convergence
of $\hat{u}$ and $\bar{u}$, we have by noting $\bar{u} \geq \hat{u}$
that

\begin{eqnarray*}
&&  |\int_{\bar{\Delta}_{\epsilon}} \bar{u}(x_{1})\log \bar{u}(x_{1})
  \mathrm{d}x_{1} -  \int_{\bar{\Delta}_{\epsilon}}  \hat{u}(x_{1})\log \hat{u}(x_{1})
  \mathrm{d}x_{1} |\\
&\leq& \int_{\bar{\Delta}_{\epsilon}} (\bar{u}(x_{1})- \hat{u}(x_{1}))|\log \bar{u}(x_{1})| \mathrm{d}x_{1} +\int_{\bar{\Delta}_{\epsilon}} \hat{u}(x_{1})(\log \bar{u}(x_{1})- \log \hat{u}(x_{1})) \mathrm{d}x_{1} \\
&:=& I_{1} + I_{2} \,.
\end{eqnarray*}

It follows from {\bf H$^{1}$)} and \eqref{(iii)} that for sufficiently small $\epsilon > 0$,
$$
\int_{\bar{\Delta}_{\epsilon}} ( \bar{u}(x_{1}) - \hat{u}(x_{1})) \mathrm{d}x_{1}
\leq \int_{\mathbb{R}^{m}} (\bar{u}(x_{1})-\hat{u}(x_{1})) \mathrm{d}x_{1} \sim o(\epsilon^{2}) \,.
$$
In addition, for all sufficient small $\epsilon > 0$ and $x \in
\bar{\Delta}_{\epsilon}$, we have by Lemma ~\ref{projbounds} that
$\bar{u}(x) < \epsilon^{-(2n+2)}$ and by the WKB expansion  of $u$
within $\bar{\Delta}_{\epsilon}$ that $\bar{u} \geq \hat{u} \sim e^{-\epsilon^{-2/5}} >
e^{-\epsilon^{-1/2}}$. Therefore $|\log
\bar{u}| < \max\{ -(2n+2)\log \epsilon, \epsilon^{-1/2} \} = \epsilon^{-1/2}$ for sufficiently small $\epsilon$. Thus
$I_{1} \sim O(\epsilon^{3/2})$. Since $\log (1+x)
\leq x$ for $x \geq 0$, we also have
$$
  I_{2} = \int_{\bar{\Delta}_{\epsilon}} \hat{u}(x_{1})\log( 1 + \frac{\bar{u}(x_{1}) - \hat{u}(x_{1})}{\hat{u}(x_{1})})
  \mathrm{d}x_{1} \leq \int_{\bar{\Delta}_{\epsilon}} (\bar{u}(x_{1}) - \hat{u}(x_{1}))  \mathrm{d}x_{1} \sim
  o(\epsilon^{2}) \,.
$$
Therefore
$$
  \lim_{\epsilon\rightarrow 0} |\int_{\bar{\Delta}_{\epsilon}} \hat{u}(x_{1})\log \hat{u}(x_{1}) \mathrm{d}x_{1} - \int_{\mathbb{R}^{m}} \bar{u}(x_{1})\log
  \bar{u} (x_{1})\mathrm{d}x_{1} | = 0 \,.
$$

It follows from Theorem ~\ref{degfixpt} that the
multivariate mutual information of system with stable equilibrium $x_{0}$
can be calculated explicitly to yield \eqref{degfix}.

 b) By the definition of degeneracy, $\mathcal{D}_{\sigma}$ is
continuously dependent on $J$ if for any coordinate decomposition
$\mathbb{R}^{n} = I_{1} \oplus I_{2} \oplus \mathcal{O}$, the limit $\lim_{\epsilon
  \rightarrow 0} MI(I_{1}; I_{2}; \mathcal{O})$ continuously depends on $J$.

\medskip

For any matrix $M \in \mathbb{R}^{n\times n}$, we denote
$\mathrm{vec}(M)$ as the vector in $\mathbb{R}^{n^{2}}$ obtained by
stacking the columns of matrix $M$. Lyapunov equation \eqref{Lyeqn} can be rewritten as
\begin{equation}
   \label{kron}
   (I - \mathrm{Kron}(J^{\top},J^{\top}))\mathrm{vec}(S) = -\mathrm{vec}(A) \,,
\end{equation}
where $\mathrm{Kron}(J^{\top},J^{\top})$ is the Kronecker product
(For more detail, see \cite{graham1982kronecker} ). Then it is easy
to see that the solution $\mathrm{vec}(S)$ continuously depends on the
Jacobian matrix $J$. Thus $S$ continuously depends on $J$.
\end{proof}




\begin{rem}
It is known that a large number of
chemical reaction networks admit unique stable equilibriums \cite{feinberg1987chemical, feinberg1979lectures,
  anderson2005stochastic, anderson2010dynamics, anderson2011proof, feinberg1995existence}. Hence the above theorem concerning
degeneracy near equilibrium is more applicable to these biological/chemical
reaction network models.

Different from systems with non-degenerate attractor, the
$\sigma$-degeneracy of systems with stable equilibrium strongly
depend on the noise matrix $\sigma(x)$. The distribution of the
perturbed system is approximately determined by the solution of
Lyapunov equation \eqref{Lyeqn}. Denote
$$\mathcal{L}_{J}S = -J^{\top}S - JS^{\top} $$
as the Lyapunov operator. It follows from
\cite{Rajendra1997note} that $\mathcal{L}_{J}$ is an invertible
operator in
the space of positive definite matrices provided that matrix $J$ is stable (all eigenvalues of $J$ has negative real parts). This
means that one can always find some perturbation matrix $\sigma(x)$
such that the resulting system has positive $\sigma-$degeneracy.

\end{rem}

\begin{ex}[Enzyme kinetic network]
Consider the following enzyme kinetic network for a substrate
competition model \eqref{enzyme1}, in which two
substrates $S_{1}$ and $S_{2}$ are catalyzed by a single enzyme
$E$. The enzyme can bind its substrates and form enzyme-substrate
complexes ($SE_{1}$ and $SE_{2}$). Products of the two enzyme-catalyzed
reactions are $P_{1}$ and $P_{2}$, respectively. Substrate competitions can be
found in many cellular processes, gene expression networks, and signal
pathway networks \cite{schauble2013effect, pocklington1969competition,
  hargreaves1988neutral, kim2011gene, jores2003essential}.

In this example, we assume substrates, the enzyme, and 
that products exchange with external environment at certain rates.
 More precisely, we consider the following reaction
equations:
\begin{equation}
  \label{enzyme1}
\begin{split}
&\emptyset  \xrightarrow{k_{1}}  S_{1}, \quad\emptyset  \xrightarrow{k_{2}} S_{2}
\\
 &S_{1} + E \xrightleftharpoons[k_3^{-1}]{k_3}   S_{1}E
 \xrightarrow{k_{4}} P_{1} + E \\
 &S_{2} + E \xrightleftharpoons[k^{-1}_5]{k_5}   S_{2}E
 \xrightarrow{k_{6}} P_{2} + E \\
& P_{1} \xrightarrow{k_{7}}  \emptyset, \quad P_{2} \xrightarrow{k_{8}}  \emptyset\\
 &E \xrightleftharpoons[k_{10}]{k_{9}} \emptyset
\end{split}
\end{equation}

Let $x_{1}, \cdots, x_{7}$ be the concentration of $S_{1}$, $S_{2}$,
$E$, $S_{1}E$, $S_{2}E$, $P_{1}$, and $P_{2}$, respectively. The
mass-action equations of this enzyme kinetic model 
read
\begin{equation}
  \label{mass1}
\begin{split}
 x_{1}'& =  k_{1} + k_{3}^{-1}x_{4} - k_{3} x_{1} x_{3} \\
x_{2}' &=  k_{2} + k_{5}^{-1}x_{5} - k_{5} x_{2} x_{3} \\
x_{3}' &=  k_{10} - k_{9}x_{3} - k_{3}x_{1}x_{3} - k_{5}x_{2}x_{3} +
(k_{3}^{-1} + k_{4})x_{4} + (k_{5}^{-1} + k_{6})x_{5}\\
x_{4}' &=  k_{3} x_{1}x_{3} - (k_{4} + k_{3}^{-1}) x_{4}\\
x_{5}' &=  k_{5} x_{2}x_{3} - (k_{6} + k_{5}^{-1}) x_{5}\\
x_{6}' &=  k_{4} x_{4} - k_{7} x_{6}\\
x_{7}'&= k_{6} x_{5} - k_{8} x_{7}
\end{split}
\end{equation}

By the deficiency zero theorem \cite{feinberg1987chemical}, it is
easy to check that  the system \eqref{mass1} admit a unique stable equilibrium $\mathbf{x}_{*}$.
Therefore, one can apply Theorem \ref{degfixpt}  to explicitly
calculate the degeneracy  of system \eqref{mass1}. Let $I_{1} = \{ S_{1} \}$, $I_{2} = \{ S_{2} \}$ be
the input sets and $\mathcal{O} = \{P_{1}, P_{2}\}$ be the output
set. We choose parameters $k_{1} = 5, k_{2} = 10, k_{3} = 20$,
$k_{4} = 5, k_{5} = 10, k_{6} = 10$, $ k_{7} = 1, k_{8} = 1, k_{9} =
2.5$, $k_{10} = 3$, and $k_{3}^{-1} = k_{5}^{-1} = 0.1$. Although
these parameters are artificially chosen, we remark that the
qualitative result in this example holds with other parameters.

With the parameters chosen above, we have
$$
MI_{0} := \lim_{\epsilon\rightarrow 0} MI(I_{1}; I_{2}; \mathcal{O} )
= 0.0646 \,.
$$
This implies a weak but
positive degeneracy $\mathcal{D}( \mathcal{O})$ of this enzyme kinetic
network. Heuristically,
this means different components of the network input, i.e., $S_{1}$ and
$S_{2}$, can perform certain common function at the output set
$\{P_{1}, P_{2} \}$. In addition, by Theorem \ref{deg2com}, this
system has positive complexity. ( With the parameters above, the
mutual information between $I_{1}$ and $I_{2}$ is $0.5338$. )

The degeneracy of this simple enzyme kinetic network can be enhanced
in the following two ways.

\begin{itemize}
  \item Assume products $P_{1}$ and $P_{2}$ are merged into one species
    $P$ and let $\mathcal{O} = \{ P \}$. With the same set of parameters (the rate of $P
    \rightarrow \emptyset$ becomes $k_{7} + k_{8}$), we observe $16.72 \%$
    increase of $MI_{0}$. This result coincides with the conceptual
    interpretation that by merging two product species into one, a
    small interruption on a subset of the network input gives less impact to
    the output (higher degeneracy).
\item Assume substrates $S_{1}$ and $S_{2}$ can be converted into each
  other with rates $k_{a}$ and $k_{b}$
$$
  S_{1} \xrightleftharpoons[k_{b}]{k_{a}} S_{2}
$$
while other reactions and parameters are as in the original setting.
Then the degeneracy increases with suitable $k_{a}$ and $k_{b}$. For
example, an $86.48 \%$ increase of $MI_{0}$ is observed with $k_{a}
= k_{b} = 5$. See Figure \ref{enzymedeg} for values of $MI_{0}$ with
varying $k_{a}$ and $k_{b}$. Conceptually, this means 
that the impact of a small interruption on a subset of the network
input can be reduced (i.e., higher degeneracy) by adding
interactions among the network input. We remark that a similar
numerical observation was made for the IL-4R and EpoR crosstalk
model in \cite{li2012quantification}. We conjecture that under
certain conditions, adding interactions among the input components
of a mass-action network will increase its degeneracy.
\end{itemize}

\begin{figure}[htbp]
\centerline{\includegraphics[height = 9cm]{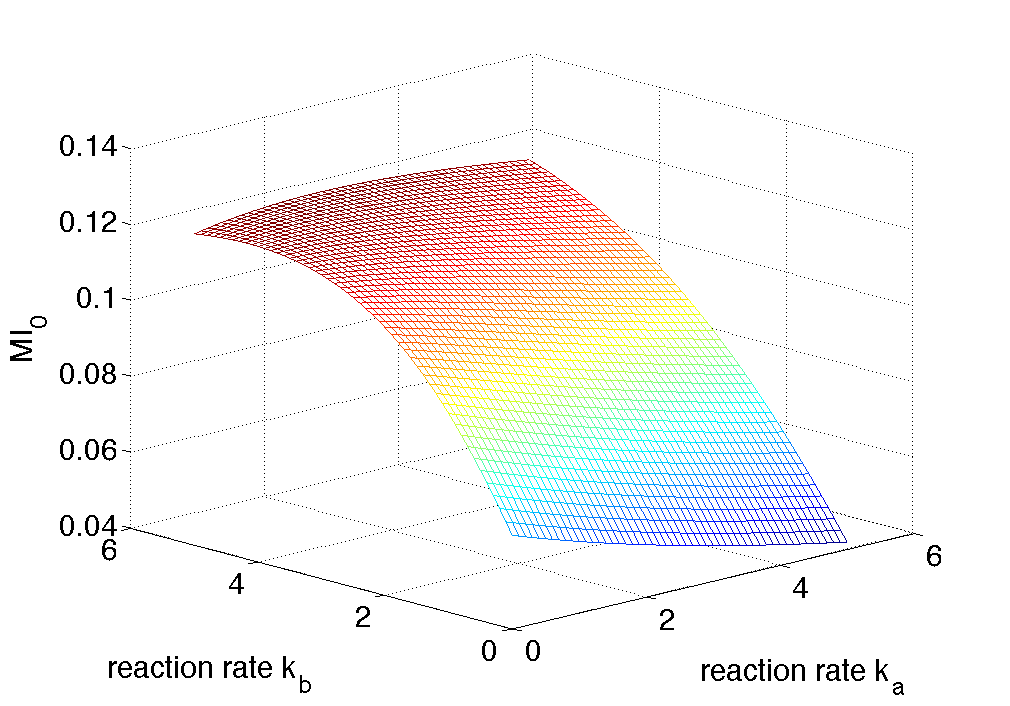}}
\caption{\label{enzymedeg} The change of $MI_{0}$ with varying $k_{a}$
and $k_{b}$. }
\end{figure}

\end{ex}

\bibliography{myref}
\bibliographystyle{plain}

\end{document}